\documentclass[a4paper]{amsart}

\usepackage[all]{xy}
\usepackage{amsmath}
\usepackage{amssymb}
\usepackage{amsthm}
\usepackage{latexsym}
\usepackage{enumerate}
\usepackage{prettyref}
\usepackage{hyperref}
\usepackage{bm}
\usepackage{mathrsfs}

\newtheorem{theorem}{Theorem}[section]
\newrefformat{thm}{\hyperref[{#1}]{Theorem~\ref*{#1}}}
\newtheorem{definition}[theorem]{Definition}
\newrefformat{def}{\hyperref[{#1}]{Definition~\ref*{#1}}}
\newtheorem{lemma}[theorem]{Lemma}
\newrefformat{lem}{\hyperref[{#1}]{Lemma~\ref*{#1}}}
\newtheorem{proposition}[theorem]{Proposition}
\newrefformat{prop}{\hyperref[{#1}]{Proposition~\ref*{#1}}}
\newtheorem{corollary}[theorem]{Corollary}
\newrefformat{cor}{\hyperref[{#1}]{Corollary~\ref*{#1}}}
\newtheorem{remark}[theorem]{Remark}
\newrefformat{rem}{\hyperref[{#1}]{Remark~\ref*{#1}}}

{
\newtheorem{examplecore}[theorem]{Example}}
\newrefformat{ex}{\hyperref[{#1}]{Example~\ref*{#1}}}

\newenvironment{proofof}[1]{\vspace{.2cm}\noindent\textsc{Proof of
    #1:}}{\hspace*{\fill} $\blacksquare$\par\vspace{.1cm}}
\newenvironment{example}{\begin{examplecore}}{\hspace*{\fill}
$\square$\par\vspace{.1cm}\end{examplecore}}

\newcommand{\op}{\operatorname}

\newcommand{\ret}{\mathrm{r\acute{e}t}}
\newcommand{\Nis}{\mathrm{Nis}}

\begin{document}

\title{Chow--Witt rings of Grassmannians}  

\author{Matthias Wendt}

\date{March 2020}

\address{Matthias Wendt, Fachgruppe Mathematik/Informatik, Bergische Universit\"at Wuppertal, Gaussstrasse 20, 42119 Wuppertal, Germany}
\email{m.wendt.c@gmail.com}

\subjclass[2010]{}

\begin{abstract}
We complement the previous computation of the Chow--Witt rings of classifying spaces of special linear groups by an analogous computation for the general linear groups. This case involves discussion of non-trivial dualities. The computation proceeds along the lines of the classical computation of the integral cohomology of ${\op{B}}\op{O}(n)$ with local coefficients, as done by \v Cadek. The computations of Chow--Witt rings of classifying spaces of $\op{GL}_n$ are then used to compute the Chow--Witt rings of the finite Grassmannians. As before, the formulas are close parallels of the formulas describing integral cohomology rings of real Grassmannians. 
\end{abstract}

\maketitle
\setcounter{tocdepth}{1}
\tableofcontents

\section{Introduction}

The present paper is a sequel to \cite{chow-witt} which computed the Chow--Witt rings of classifying spaces of the symplectic and special linear groups. The present paper  provides a similar computation of the total Chow--Witt ring of ${\op{B}}\op{GL}_n$, essentially by a combination and extension of the techniques developed in \cite{chow-witt} and \v Cadek's computation of cohomology of ${\op{B}}\op{O}(n)$ with twisted coefficients in \cite{cadek}. Once the characteristic classes for vector bundles and their relations are known, we can also compute the Chow--Witt rings of finite Grassmannians. It turns out that the formulas describing the Chow--Witt rings of Grassmannians are direct analogues of the classical formulas for integral cohomology of real Grassmannians. 

\subsection{Chow--Witt rings of infinite Grassmannians}
We first formulate the results for the infinite Grassmannian ${\op{B}}\op{GL}_n$. In this case, there are essentially two dualities to consider, the trivial duality and the nontrivial one given by the determinant $\det\gamma_n^\vee$ of the universal rank $n$ bundle $\gamma_n$. The following result describes the total Chow--Witt ring of ${\op{B}}\op{GL}_n$, cf. Theorems~\ref{thm:glnin} and \ref{thm:glnchw} and Proposition~\ref{prop:kerpartial}. 

\begin{theorem}
\label{thm:main1}
Let $F$ be a perfect field of characteristic $\neq 2$. 
\begin{enumerate}
\item The following square, induced from the description of the Milnor--Witt K-theory sheaf, is a cartesian square of $\mathbb{Z}\oplus\mathbb{Z}/2\mathbb{Z}$-graded $\op{GW}(F)$-algebras: 
\[
\xymatrix{
\widetilde{\op{CH}}^\bullet({\op{B}}\op{GL}_n,\mathcal{O}\oplus\det \gamma_n^\vee) \ar[rr] \ar[d] && \ker\partial_{\mathscr{O}}\oplus\ker\partial_{\det\gamma_n^\vee} \subseteq \op{CH}^\bullet({\op{B}}\op{GL}_n)^{\oplus 2} \ar[d]^{\bmod 2} \\
\op{H}^\bullet_{\Nis}({\op{B}}\op{GL}_n,\mathbf{I}^\bullet\oplus\mathbf{I}^\bullet(\det\gamma_n^\vee))\ar[rr]_\rho && \op{Ch}^\bullet({\op{B}}\op{GL}_n)^{\oplus 2}.
}
\]
Here $\det\gamma_n$ is the determinant of the tautological rank $n$ bundle on ${\op{B}}\op{GL}_n$,  $\op{Ch}:=\op{CH}/2$ denotes mod 2 Chow groups, 
\[
\partial_{\mathscr{L}}\colon\op{CH}^\bullet({\op{B}}\op{GL}_n) \to\op{Ch}^\bullet({\op{B}}\op{GL}_n) \xrightarrow{\beta_{\mathscr{L}}} \op{H}^{\bullet+1}({\op{B}}\op{GL}_n, \mathbf{I}^{\bullet+1}(\mathscr{L}))
\]
is the (twisted) integral Bockstein operation. 
\item We have 
\[
\ker\partial_{\mathscr{O}}=\mathbb{Z}[\op{c}_i^2, 2\op{c}_{i_1}\cdots\op{c}_{i_k}, {\op{c}_1}{\op{c}_{2i}}+{\op{c}_{2i+1}},{\op{c}_1}{\op{c}_n}]\subseteq\mathbb{Z}[\op{c}_1,\dots,\op{c}_n]\cong \op{CH}^\bullet({\op{B}}\op{GL}_n),
\]
\[
\ker\partial_{\det\gamma_n^\vee}=\mathbb{Z}[\op{c}_{2i+1}, 2\op{c}_{2i_1}\cdots\op{c}_{2i_k},{\op{c}}_{2i}^2, \op{c}_n]\subseteq\mathbb{Z}[\op{c}_1,\dots,\op{c}_n]\cong \op{CH}^\bullet({\op{B}}\op{GL}_n).
\]
\item The cohomology ring $\op{H}^\bullet_{\Nis}({\op{B}}\op{GL}_n,\mathbf{I}^\bullet\oplus\mathbf{I}^\bullet(\det\gamma_n^\vee))$ can be explicitly identified as the $\mathbb{Z}\oplus\mathbb{Z}/2\mathbb{Z}$-graded-commutative $\op{W}(F)$-algebra 
\[
\op{W}(F)\left[\op{p}_2,\op{p}_4,\dots,\op{p}_{[(n-1)/2]},\op{e}_n, \{\beta_J\}_J, \{\tau_J\}_J, \tau_\emptyset\right]
\]
modulo the following relations:
\begin{enumerate}
\item $\op{I}(F) \beta_J=\op{I}(F)\tau_J=\op{I}(F)\tau_\emptyset=0$.
\item If $n=2k+1$ is odd, $\op{e}_{2k+1}=\tau_{\{k\}}$. 
\item For two index sets $J$ and $J'$ where $J'$ can be empty, we have 
\begin{eqnarray*}
\beta_J\cdot \beta_{J'}&=&\sum_{k\in J} \beta_{\{k\}}\cdot \op{p}_{(J\setminus\{k\})\cap J'}\cdot \beta_{\Delta(J\setminus\{k\},J')}\\
\beta_J\cdot \tau_{J'}&=&\sum_{k\in J}\beta_{\{k\}}\cdot \op{p}_{(J\setminus\{k\})\cap J'}\cdot \tau_{\Delta(J\setminus\{k\},J')}\\
\tau_J\cdot \beta_{J'}&=&\beta_J\cdot \tau_{J'} + \tau_\emptyset \cdot \op{p}_{J\cap J'}\cdot \beta_{\Delta(J,J')}\\
\tau_J\cdot \tau_{J'}&=&\beta_J\cdot \beta_{J'}+\tau_\emptyset\cdot \op{p}_{J\cap J'}\cdot \tau_{\Delta(J,J')}.
\end{eqnarray*}
where we set $\op{p}_A=\prod_{i=1}^l \op{p}_{a_i}$ for an index set $A=\{a_1,\dots,a_l\}$.
\end{enumerate}

In the above, the $\op{p}_{2i}$ are even Pontryagin classes in degree $(4i,0)$, $\op{e}_n$ is the Euler class in degree $(n,1)$. For the classes $\beta_J$ and $\tau_J$, the index set $J$ runs through the sets $\{j_1,\dots,j_l\}$ of natural numbers with $0<j_1<\dots<j_l\leq[(n-1)/2]$ and we set
\[
\beta_J=\beta_{\mathscr{O}}(\overline{\op{c}}_{2j_1}\cdots \overline{\op{c}}_{2j_l}), \textrm{ and } \tau_J=\beta_{\det\gamma_n^\vee}(\overline{\op{c}}_{2j_1}\cdots \overline{\op{c}}_{2j_l}).
\]
For the index set $J=\{j_1,\dots,j_l\}$, the degree of $\beta_J$  is $(1+2\sum_{i=1}^l j_i,0)$ and the degree of $\tau_J$ is $(1+2\sum_{i=1}^l j_i,1)$. 
\item The reduction morphism $\rho$ is given by 
\[
\op{p}_{2i}\mapsto \overline{\op{c}}_{2i}^2, \quad \op{e}_n\mapsto\overline{\op{c}}_n \quad \textrm{ and } \quad \beta_{\mathscr{L}}(\overline{\op{c}}_{2j_1}\cdots\overline{\op{c}}_{2j_l}) \mapsto \op{Sq}^2_{\mathscr{L}}(\overline{\op{c}}_{2j_1}\cdots\overline{\op{c}}_{2j_l}). 
\] 
Under the homomorphism $\widetilde{\op{CH}}^\bullet({\op{B}}\op{GL}_n,\mathcal{O}) \to\op{CH}^\bullet({\op{B}}\op{GL}_n)$, the Chow--Witt-theoretic Pontryagin class maps as
\[
\op{p}_i\mapsto (-1)^i\op{c}_i^2+2\sum_{j=\max\{0,2i-n\}}^{i-1}(-1)^j\op{c}_j\op{c}_{2i-j}. 
\]
\end{enumerate}
\end{theorem}

It can be shown that formulas similar to the above description of $\mathbf{I}^\bullet$-cohomology are true for real-\'etale cohomology, 
but as algebra over $\op{H}^0_{\ret}(F,\mathbb{Z})\cong\op{colim}_n\op{I}^n(F)$. For $F=\mathbb{R}$, the real cycle class map induces an isomorphism  
\[
\op{H}^\bullet({\op{B}}\op{GL}_n,\mathbf{I}^\bullet\oplus\mathbf{I}^\bullet(\det\gamma_n^\vee))\xrightarrow{\cong} \op{H}^\bullet({\op{B}}\op{O}(n),\mathbb{Z}\oplus\mathbb{Z}^{\op{t}})
\]
where the target was computed in \cite{cadek}, and it sends algebraic characteristic classes to their topological counterparts. For this result and a discussion of the required compatibilities e.g. between localization sequences and the real cycle class maps, cf.~\cite{4real}.

\subsection{Chow--Witt rings of finite Grassmannians}

The second point of the paper is to provide a computation of the Chow--Witt rings of the finite Grassmannians $\op{Gr}(k,n)$. The full description is even longer than the description of the Chow--Witt ring of ${\op{B}}\op{GL}_n$, so we will only give pointers to the main results in the text. First, the Chow--Witt ring is again given in terms of a cartesian square combining the kernels of integral Bockstein maps with $\mathbf{I}^\bullet$-cohomology, cf. Theorem~\ref{thm:chowwittgrass}. The $\mathbf{I}^\bullet$-cohomology of $\op{Gr}(k,n)$ can be described as follows: the characteristic classes of the tautological rank $k$ subbundle $\mathscr{S}_k$ and the tautological rank $(n-k)$ quotient bundle $\mathscr{Q}_{n-k}$ generate the $\mathbf{I}^\bullet$-cohomology, except in the case where $k(n-k)$ is odd in which we have a new class $\op{R}$ in degree $n-1$. They naturally satisfy the relations in the $\mathbf{I}^\bullet$-cohomology of ${\op{B}}\op{GL}_k$ and ${\op{B}}\op{GL}_{n-k}$, and they also satisfy the relations which are consequences of the Whitney sum formula for the extension
\[
0\to \mathscr{S}_k\to\mathscr{O}^{\oplus n}\to \mathscr{Q}_{n-k}\to 0.
\]
There are a few further relations involving the potential class $\op{R}$. All these statements are established in Theorem~\ref{thm:ingrass}. The reduction morphisms
\[
\op{H}^\bullet(\op{Gr}(k,n),\mathbf{I}^\bullet(\mathscr{L}))\to \op{Ch}^\bullet(\op{Gr}(k,n))
\]
are described in Proposition~\ref{prop:inredgrass}. Except for the new fact that $\op{R}\mapsto \overline{\op{c}}_{k-1}\overline{\op{c}}_{n-k}^\perp$, the description of the reduction morphisms follows directly from the ones for ${\op{B}}\op{GL}_k$ and ${\op{B}}\op{GL}_{n-k}$. This also provides a description of the kernel of the integral Bockstein maps, cf. Theorem~\ref{thm:chowwittgrass}. Again, similar formulas would be true in real-\'etale cohomology, and for $F=\mathbb{R}$ the above description recovers exactly the integral cohomology of the real Grassmannians $\op{Gr}_k(\mathbb{R}^n)$ (with local coefficients), cf.~\cite{4real}. 

\subsection{Decomposition of $\mathbf{I}$-cohomology}
The present paper is a significantly revised version of its predecessor. While the previous version established the results mostly following the proof strategy of \cite{chow-witt} fairly closely, the revised proofs follow a different strategy. The key new insight arises from a decomposition of $\mathbf{I}$-cohomology in the B\"ar sequence
\[
0\to \op{Im}\beta_{\mathscr{L}}(X)\to \op{H}^q(X,\mathbf{I}^q(\mathscr{L}))\to \op{H}^q(X,\mathbf{W}(\mathscr{L}))\to 0.
\]
The cohomology with coefficients in the sheaf $\mathbf{W}$ of Witt rings is a theory in which $\eta$ is invertible, and much more amenable to long exact sequence calculations than $\mathbf{I}$-cohomology, cf. work of Ananyevskiy \cite{ananyevskiy}. Moreover, if $\mathbf{W}$-cohomology is free, the above sequence splits. In that case, the reduction morphism $\rho\colon \op{H}^q(X,\mathbf{I}^q(\mathscr{L}))\to \op{Ch}^q(X)$ is injective on the image of $\beta_{\mathscr{L}}$, hence the torsion classes in $\op{Im}\beta_{\mathscr{L}}$ can be computed from the knowledge of the Steenrod squares $\op{Sq}^2_{\mathscr{L}}$ on the mod 2 Chow ring. This way, the computation of $\mathbf{I}$-cohomology splits into two significantly easier parts, the computation of $\mathbf{W}$-cohomology which can be done by the same methods as calculations of rational cohomology of real Grassmannians, cf. \cite{sadykov} and \cite{milnor:stasheff}, and the computation of $\op{Im}\beta$ which only requires knowledge of the mod 2 Chow theory. The freeness of $\mathbf{W}$-cohomology which implies the above decomposition of $\mathbf{I}$-cohomology can therefore be seen as the algebraic analogue of the statement that ``all torsion in the cohomology of real Grassmannians is 2-torsion''. This gets rid of  problems as in Remark~2.2 or Remark~7.2 of \cite{chow-witt}.  Moreover, the proof of freeness of $\mathbf{W}$-cohomology and therefore the torsion statement is significantly easier than in topology (where it is not immediately clear that integral cohomology modulo the image of the integral Bockstein maps is even a cohomology theory). The $\op{Im}\beta$-$\mathbf{W}$-decomposition of $\mathbf{I}$-cohomology will be a useful tool for a number of further computations (where the real topological counterparts have only 2-torsion), such as classifying spaces of orthogonal groups and flag varieties.

The shorter, alternative way to describe the structure of the $\mathbf{I}$-cohomology of ${\op{B}}\op{GL}_n$ or $\op{Gr}(k,n)$ (at least additively) is then the following. The $\mathbf{I}$-cohomology splits as a direct sum of the image of $\beta_{\mathscr{L}}$, which is a 2-torsion group with the same structure as in the integral cohomology of the real Grassmannians, and the $\mathbf{W}$-cohomology, which is a free $\op{W}(F)$-algebra having the same presentation as the rational cohomology of the real Grassmannians. The multiplication on the torsion part can be described completely by reduction to mod 2 Chow theory where we have the classical formulas from Schubert calculus. Conceptual descriptions for the multiplication can be found in \cite{casian:stanton} (with an interesting link to representation theory of real Lie groups) and \cite{casian:kodama} (explicitly in terms of signed Young diagrams), cf. also the discussion of checkerboard fillings for Young diagrams to compute $\op{Sq}^2_{\mathscr{L}}$ in \cite{schubert}. The description of Chow--Witt rings of finite Grassmannians is used in a sequel \cite{schubert} to develop an oriented Schubert calculus which allows to establish arithmetic refinements of classical Schubert calculus. 

\emph{Structure of the paper:}
We provide a recollection on relevant statements from Chow--Witt theory, in particular the twisted Steenrod squares, in Section~\ref{sec:recall}. The relevant characteristic classes for vector bundles are recalled in Section~\ref{sec:vbclasses}, where we also formulate the main structural results on the Chow--Witt ring of ${\op{B}}\op{GL}_n$. The inductive computation of the $\mathbf{I}^\bullet$-cohomology is done in Section~\ref{sec:inproofs}. The results on Chow--Witt rings of finite Grassmannians are formulated in Section~\ref{sec:chowwittgrass} and the proofs are given in Section~\ref{sec:grassproofs}. 

\emph{Acknowledgements:} I would like to thank Jens Hornbostel and Kirsten Wickelgren for stimulating discussions about this and related topics without which this paper wouldn't exist. Various discussions with Jens Hornbostel, Thomas Hudson, Lorenzo Mantovani, Toan Manh Nguyen and Konrad Voelkel clarified that the use of the $\op{Im} \beta$-$\mathbf{W}$-decomposition of $\mathbf{I}$-cohomology could improve structural statements and streamline proofs. 

\section{Recollection on Chow--Witt rings}
\label{sec:recall}

Throughout the article, we consider a perfect base field $F$ of characteristic $\op{char}(F)\neq 2$. All the relevant cohomology groups will be Nisnevich cohomology groups, i.e., Ext-groups between Nisnevich sheaves on the small site of a smooth scheme. 

Since this is a sequel to \cite{chow-witt}, most of the general facts concerning Chow--Witt rings relevant for the computation in the present paper can already be found in the discussion of \cite[Section 2]{chow-witt} (or, of course, in the original literature, cf. loc.cit. for references). The same applies to the general discussion of Chow--Witt rings of classifying spaces; all the statements relevant for the present paper can be found in \cite[Section 3]{chow-witt}. We freely use the definitions, facts and notation from \cite{chow-witt}. 

\subsection{Twisted coefficients and cohomology operations}
What has to be discussed in slightly more detail is the use of twisted coefficients in Chow--Witt groups resp. $\mathbf{I}^\bullet$-cohomology which was only mentioned in passing in \cite{chow-witt}. If $\mathscr{L}$ is a line bundle on a smooth scheme $X$, then there are twisted sheaves $\mathbf{I}^n(\mathscr{L})$ and $\mathbf{K}^{\op{MW}}_n(\mathscr{L})$. For the construction as well as a description of Gersten-type complexes computing $\op{H}^n_{\op{Nis}}(X,\mathbf{I}^n(\mathscr{L}))$ and $\op{H}^n_{\op{Nis}}(X,\mathbf{K}^{\op{MW}}_n(\mathscr{L}))\cong \widetilde{\op{CH}}^n(X,\mathscr{L})$, see \cite[Section 10]{fasel:memoir} and \cite[Section 2]{AsokFaselEuler}. In particular, Theorem 2.3.4 of \cite{AsokFaselEuler} provides an identification of the definition of twisted Chow--Witt groups in \cite{fasel:memoir} with the Nisnevich cohomology of the twisted Milnor--Witt K-theory sheaves. If $\mathscr{L}$ and $\mathscr{N}$ are two line bundles on $X$, then there are canonical isomorphisms
\[
\widetilde{\op{CH}}^\bullet(X,\mathscr{N})\cong \widetilde{\op{CH}}^\bullet(X,\mathscr{L}^2\otimes\mathscr{N}). 
\]

The twisted versions of Chow--Witt groups and $\mathbf{I}^n$-cohomology have functorial pullbacks, pushforwards and a localization sequence (where the cohomology of the closed subscheme appears with twist by the normal bundle of the inclusion). Formulations and references to the relevant literature can all be found in \cite[Section 2.1]{chow-witt}. 

We also need to discuss twisted analogues of the facts on cohomology operations discussed in \cite[Section 2.3]{chow-witt}. If $X$ is a smooth scheme and $\mathscr{L}$ is a line bundle on $X$, we can twist the exact sequence of fundamental ideals by $\mathscr{L}$ to get an exact sequence of strictly $\mathbb{A}^1$-invariant Nisnevich sheaves of abelian groups
\[
0\to \mathbf{I}^{n+1}(\mathscr{L})\to \mathbf{I}^n(\mathscr{L})\to \mathbf{K}^{\op{M}}_n/2\to 0. 
\]
This is analogous to the topological exact sequence $0\to \mathbb{Z}^{\op{t}}\to \mathbb{Z}^{\op{t}}\to \mathbb{Z}/2\mathbb{Z}\to 0$ for a local system $\mathbb{Z}^{\op{t}}$ with fiber $\mathbb{Z}$. Associated to the previous exact sequence of Nisnevich sheaves, we get a twisted version of the B\"ar sequence:
\[
\to \op{H}^n(X,\mathbf{I}^{n+1}(\mathscr{L})) \xrightarrow{\eta} \op{H}^n(X,\mathbf{I}^{n}(\mathscr{L})) \xrightarrow{\rho} \op{Ch}^n(X) \xrightarrow{\beta_{\mathscr{L}}} \op{H}^{n+1}(X,\mathbf{I}^{n+1}(\mathscr{L}))\to
\]
The connecting map $\beta_{\mathscr{L}}$ for this sequence is a Chow--Witt analogue of the Bockstein operation twisted by a local system in classical algebraic topology (for a discussion in the context of cohomology of ${\op{B}}\op{O}(n)$, cf. \cite{cadek}). 

\begin{remark}
A funny side remark on some of the differences between $\op{Ch}^n(X)$ and mod 2 singular cohomology. By the universal coefficient formula, mod 2 singular cohomology $\op{H}^n$ in general contains mod 2 reductions of integral classes in $\op{H}^n$ \emph{as well as} classes related to $2$-torsion classes in $\op{H}^{n-1}$. This is not true for $\op{Ch}^n(X)$, viewed as mod 2 reduction of $\op{CH}^n(X)$ -- by definition all classes in $\op{Ch}^n(X)$ are simply mod 2 reductions of $\op{CH}^n(X)$. However, the B\"ar sequence encodes a behaviour of $\op{Ch}^n(X)$ exactly analogous to mod 2 singular cohomology: there are some classes which lift to integral cohomology $\op{H}^n(X,\mathbf{I}^n)$, and some classes which don't (because they have nontrivial images under the Bockstein operation). 
\end{remark}

For a line bundle $\mathscr{L}$ on a smooth scheme $X$, the twisted Bockstein map
\[
\beta_{\mathscr{L}}\colon\op{Ch}^n(X)\to \op{H}^{n+1}(X,\mathbf{I}^{n+1}(\mathscr{L}))
\]
can be used to define twisted versions of integral Stiefel--Whitney classes analogous to those defined in \cite{fasel:ij}. The composition with the reduction morphism $\rho\colon\op{H}^{n+1}(X,\mathbf{I}^{n+1}(\mathscr{L}))\to \op{Ch}^{n+1}(X)$ has been identified in \cite[Theorem 3.4.1]{AsokFaselSecondary}. This is a twisted version of Totaro's identification \cite[Theorem 1.1]{totaro:witt} and a Chow--Witt version of \cite[Lemma 2]{cadek}. 

\begin{proposition}
\label{prop:twistedsq2}
Let $X$ be a smooth scheme and $\mathscr{L}$ be a line bundle over $X$. Denote by $\beta_{\mathscr{L}}\colon\op{Ch}^i(X)\to\op{H}^{i+1}(X,\mathbf{I}^{i+1}(\mathscr{L}))$ the twisted Bockstein map. Then for all $x\in\op{Ch}^i(X)$ we have
\[
\rho\circ\beta_{\mathscr{L}}(x)=\overline{\op{c}}_1(\mathscr{L})\cdot x+\op{Sq}^2(x)=:\op{Sq}^2_{\mathscr{L}}(x).
\]
\end{proposition}

\subsection{Oriented intersection product and total Chow--Witt ring}
The oriented intersection product for the Chow--Witt ring has the form 
\[
\widetilde{\op{CH}}^i(X,\mathscr{L}_1)\times \widetilde{\op{CH}}^j(X,\mathscr{L}_2)\to \widetilde{\op{CH}}^{i+j}(X,\mathscr{L}_1\otimes\mathscr{L}_2);
\]
there is a similar product on twisted $\mathbf{I}^\bullet$-cohomology rings. With these products, the Chow--Witt ring is a $\langle-1\rangle$-graded commutative algebra over the Grothendieck--Witt ring $\op{GW}(F)$, and $\bigoplus_n\op{H}^n(X,\mathbf{I}^n)$ is a $(-1)$-graded commutative algebra over the Witt ring $\op{W}(F)$, cf. e.g. \cite[Section 2.2]{chow-witt}.

The total Chow--Witt ring of a smooth scheme $X$ is defined by 
\[
\bigoplus_{\mathscr{L}\in\op{Pic}(X)/2}\widetilde{\op{CH}}^\bullet(X,\mathscr{L}), 
\]
cf. e.g. \cite[Definition 6.10]{fasel:chowwitt}. Strictly speaking, a total Chow--Witt ring doesn't exist because identifications $\widetilde{\op{CH}}^\bullet(X,\mathscr{L})\cong \widetilde{\op{CH}}^\bullet(X,\mathscr{N})$ for isomorphic line bundles $\mathscr{L}$ and $\mathscr{N}$ depend on the choice of isomorphism between the line bundles. However, the technical inaccuracy of neglecting such choices of isomorphisms between different representatives of isomorphism classes of line bundles can be fixed by the methods in \cite{lax:similitude}. The same goes for the total $\mathbf{I}$-cohomology ring 
\[
\bigoplus_{\mathscr{L}\in\op{Pic}(X)/2, q\in\mathbb{N}}\op{H}^q(X,\mathbf{I}^q(\mathscr{L})).  
\]

Note that $\op{Pic}({\op{B}}\op{GL}_n)\cong \mathbb{Z}$ and $\op{Pic}(\op{Gr}(k,n))\cong\mathbb{Z}$; in particular, there are only two nontrivial dualities to consider for the total Chow--Witt rings of ${\op{B}}\op{GL}_n$ and $\op{Gr}(k,n)$. For ${\op{B}}\op{GL}_n$, the nontrivial element of $\op{Pic}({\op{B}}\op{GL}_n)/2$ is given by $\det\gamma_n^\vee$, the dual of the determinant of the universal rank $n$ bundle. Note that this corresponds precisely to the well-known topological fact that there are exactly two isomorphism classes of local systems on ${\op{B}}\op{O}(n)$, the trivial one and the one for the sign representation of $\pi_1({\op{B}}\op{O}(n))\cong\mathbb{Z}/2\mathbb{Z}$ on the coefficient ring $\mathbb{Z}$.

\subsection{The fundamental square}
After having discussed all the relevant preliminaries, there are now twisted analogues of the key diagram from \cite{chow-witt}, for any line bundle $\mathscr{L}$ on $X$:
\[
\xymatrix{
&\op{CH}^n(X)\ar[r]^= \ar[d] & \op{CH}^n(X)\ar[d]^2 \\
\op{H}^n(X,\mathbf{I}^{n+1}(\mathscr{L}))\ar[r]\ar[d]_=& \widetilde{\op{CH}}^n(X,\mathscr{L})\ar[r]\ar[d]&\op{CH}^n(X)\ar[r]^(.4){\partial_{\mathscr{L}}}\ar[d]^{\bmod 2}& \op{H}^{n+1}(X,\mathbf{I}^{n+1}(\mathscr{L}))\ar[d]^=\\
\op{H}^n(X,\mathbf{I}^{n+1}(\mathscr{L}))\ar[r]_\eta& \op{H}^{n}(X,\mathbf{I}^{n}(\mathscr{L}))\ar[r]_\rho\ar[d]& \op{Ch}^n(X)\ar[r]^(.4){\beta_{\mathscr{L}}}\ar[rd]_{\op{Sq}^2_{\mathscr{L}}}\ar[d]& \op{H}^{n+1}(X,\mathbf{I}^{n+1}(\mathscr{L}))\ar[d]^\rho \\
&0\ar[r]&0&\op{Ch}^{n+1}(X)
}
\]

As already mentioned in \cite{chow-witt}, there is a twisted analogue of \cite[Proposition 2.11]{chow-witt}, which states that for $F$ a perfect field of characteristic unequal to $2$ and a smooth scheme $X$ over $F$ the canonical map 
\[
c\colon\widetilde{\op{CH}}^\bullet(X,\mathscr{L})\to \op{H}^\bullet(X,\mathbf{I}^\bullet(\mathscr{L}))\times_{\op{Ch}^\bullet(X)} \ker\partial_{\mathscr{L}}
\] 
induced from the above key square is always surjective, and is injective if $\op{CH}^\bullet(X)$ has no non-trivial $2$-torsion. This way we can determine the additive structure of twisted Chow--Witt groups; if we consider the total Chow--Witt ring (i.e., the direct sum of twisted Chow--Witt groups over $\op{Pic}(X)/2$), the fiber square also describes the oriented intersection product. The result applies, in particular, to ${\op{B}}\op{GL}_n$ and the Grassmannians $\op{Gr}(k,n)$ (or more generally flag varieties $G/P$ for reductive groups) because these are known to have 2-torsion-free Chow groups. This implies that we only need to determine the individual terms of the fiber product to get a description of the Chow--Witt ring. 

\subsection{Decomposing $\mathbf{I}$-cohomology into $\mathbf{W}$-cohomology and the image of $\beta$}
One of the features which is new and hasn't been used in either \cite{chow-witt} or the previous version of the present paper is $\mathbf{W}$-cohomology. For a smooth $F$-scheme $X$, we can consider the restriction of the Nisnevich sheaf $\mathbf{W}$ of Witt groups to the small Nisnevich site of $X$, and then take its Nisnevich cohomology $\op{H}^\bullet(X,\mathbf{W})$. As before, if $\mathscr{L}$ is a line bundle on $X$, we can consider the twisted $\mathbf{W}$-cohomology groups $\op{H}^\bullet(X,\mathbf{W}(\mathscr{L}))$. The product structure on the Witt rings induces an intersection product
\[
\op{H}^i(X,\mathbf{W}(\mathscr{L}_1))\times\op{H}^j(X,\mathbf{W}(\mathscr{L}_2))\to \op{H}^{i+j}(X,\mathbf{W}(\mathscr{L}_1\otimes\mathscr{L}_2)), 
\]
and we can consider the total $\mathbf{W}$-cohomology ring $\bigoplus_{q,\mathscr{L}\in \op{Pic}(X)/2}\op{H}^q(X,\mathbf{W}(\mathscr{L}))$ (again using \cite{lax:similitude} to make sense of this). Similar to the $\mathbf{I}$-cohomology ring, the total $\mathbf{W}$-cohomology ring is a $(-1)$-graded commutative algebra over the Witt ring $\op{W}(F)$. 

There is a morphism $(\mathbf{I}^n)_{n\in\mathbb{Z}}\to (\mathbf{W})_{n\in\mathbb{Z}}$ which in degree $n$ is given by the natural inclusion $\mathbf{I}^n\hookrightarrow\mathbf{W}$, with the usual convention of $\mathbf{I}^n=\mathbf{W}$ for $n\leq 0$. This morphism induces a $\op{W}(F)$-algebra homomorphism
\[
\bigoplus_{q,\mathscr{L}}\op{H}^q(X,\mathbf{I}^q(\mathscr{L}))\to \bigoplus_{q,\mathscr{L}}\op{H}^q(X,\mathbf{W}(\mathscr{L}))
\]
from the total $\mathbf{I}$-cohomology ring to the total $\mathbf{W}$-cohomology ring.

The relation with $\mathbf{I}$-cohomology can be made more precise. The pieces
\[
\op{H}^{i-1}(X,\mathbf{K}^{\op{M}}_{n-1}/2)\to \op{H}^i(X,\mathbf{I}^n(\mathscr{L}))\to \op{H}^i(X,\mathbf{I}^{n-1}(\mathscr{L}))\to \op{H}^{i}(X,\mathbf{K}^{\op{M}}_{n-1}/2)
\]
of the B\"ar sequence provide isomorphisms $\op{H}^i(X,\mathbf{I}^n(\mathscr{L}))\to \op{H}^i(X,\mathbf{I}^{n-1}(\mathscr{L}))$ for $i>n$, because the outer terms vanish. This can be seen from the Gersten resolution for mod 2 Milnor K-theory together with the fact that $(\mathbf{K}^{\op{M}}_{n-1}/2)_{-c}=0$ for $c>n-1$. Moreover, from the Gersten resolution for $\mathbf{I}^n$, we also see that the natural morphisms $\op{H}^i(X,\mathbf{I}^n(\mathscr{L}))\to \op{H}^i(X,\mathbf{W}(\mathscr{L}))$ are isomorphisms for $i>n$. Now with this reinterpretation, we can consider the piece of the B\"ar sequence for the boundary case $i=n$:
\[
\op{Ch}^{n-1}(X)\cong \op{H}^{n-1}(X,\mathbf{K}^{\op{M}}_{n-1}/2)\xrightarrow{\beta_{\mathscr{L}}} \op{H}^n(X,\mathbf{I}^n(\mathscr{L}))\to \op{H}^n(X,\mathbf{W}(\mathscr{L}))\to 0.
\]
In particular, $\mathbf{I}$-cohomology is a combination of $\mathbf{W}$-cohomology with the image of the Bockstein morphism $\beta$. We get a stronger splitting result if the $\mathbf{W}$-cohomology is free:

\begin{lemma}
  \label{lem:wsplit}
  Let $X$ be a smooth scheme over a field $F$ of characteristic $\neq 2$, and let $\mathscr{L}$ be a line bundle on $X$. If $\op{H}^n(X,\mathbf{W}(\mathscr{L}))$ is free as a $\op{W}(F)$-module, then we have a splitting
  \[
  \op{H}^n(X,\mathbf{I}^n(\mathscr{L}))\cong \op{Im}\beta_{\mathscr{L}}\oplus\op{H}^n(X,\mathbf{W}(\mathscr{L})).
  \]
  In this case, the reduction morphism $\rho\colon \op{H}^n(X,\mathbf{W}(\mathscr{L}))\to \op{Ch}^n(X)$ is injective on the image of $\beta_{\mathscr{L}}$. 
\end{lemma}

\begin{proof}
  The B\"ar sequence is a long exact sequence of $\op{W}(F)$-modules. The first claim follows from the piece
  \[
  \op{Ch}^{n-1}(X)\xrightarrow{\beta_{\mathscr{L}}} \op{H}^n(X,\mathbf{I}^n(\mathscr{L}))\to \op{H}^n(X,\mathbf{W}(\mathscr{L}))\to 0.
  \]
  If the last group is free as $\op{W}(F)$-module, then the sequence splits as claimed. The second claim follows from the following exact piece of the B\"ar sequence
  \[
  \op{H}^n(X,\mathbf{I}^{n+1}(\mathscr{L}))\xrightarrow{\eta} \op{H}^n(X,\mathbf{I}^n(\mathscr{L}))\xrightarrow{\rho}\op{Ch}^n(X).
  \]
  The morphism $\eta\colon \op{H}^n(X,\mathbf{I}^{n+1}(\mathscr{L}))\to \op{H}^n(X,\mathbf{I}^n(\mathscr{L}))$ has image in the submodule $\op{I}(F)\cdot \op{H}^n(X,\mathbf{I}^n(\mathscr{L}))$ and the splitting implies
  \[
  \left(\op{I}(F)\cdot\op{H}^n(X,\mathbf{I}^n(\mathscr{L}))\right)\cap\op{Im}\beta_{\mathscr{L}}=0.
  \]
  In particular, $\op{Im}\beta_{\mathscr{L}}\cap \op{Im}\eta=0$, and this implies the injectivity claim.
\end{proof}

\begin{corollary}
  Let $X$ be a smooth scheme over a field $F$ of characteristic $\neq 2$, and let $\mathscr{L}$ be a line bundle on $X$. If the total $\mathbf{W}$-cohomology ring is free as a $\op{W}(F)$-module, then the image of the maps $\beta_{\mathscr{L}}$ for $\mathscr{L}\in\op{Pic}(X)/2$ coincides exactly with the $\op{W}(F)$-torsion in $\mathbf{W}$-cohomology. In particular, the image of the maps $\beta_{\mathscr{L}}$ for $\mathscr{L}\in\op{Pic}(X)/2$ is an ideal in  the total $\mathbf{W}$-cohomology ring.
\end{corollary}



\begin{remark}
  The freeness of $\mathbf{W}$-cohomology in this lemma will play an important role in our computations. It is an algebraic replacement of the classical statement that ``all torsion in the cohomology of the Grassmannians is 2-torsion'', as formulated e.g. in \cite[Lemma 2.2]{brown}. Using the splitting in Lemma~\ref{lem:wsplit} is a different strategy than the cumbersome proofs in \cite{chow-witt} which were needed to establish that $\rho$ is injective on the image of $\beta$, cf. Remark 7.2 and the discussion before Proposition 8.6 in \cite{chow-witt}.
\end{remark}


There are two reasons why the decomposition of $\mathbf{I}$-cohomology as a direct sum of $\mathbf{W}$-cohomology and the image of $\beta$ is so effective as a computational tool. On the one hand, the image of $\beta$ is basically known in the relevant cases -- all it requires is knowledge of the Chow ring together with the action of $\op{Sq}^2$. On the other hand, computations in $\mathbf{W}$-cohomology are simpler than for $\mathbf{I}$-cohomology because the localization sequence takes the following simplified form: assume $X$ is a smooth scheme, $Z\subseteq X$ a smooth closed subscheme of pure codimension $c$ with open complement $U=X\setminus Z$, and $\mathscr{L}$ is a line bundle on $X$. Denote the inclusions by $i\colon Z\hookrightarrow X$ and $j\colon U\hookrightarrow X$, and denote by $\mathscr{N}$ the determinant of the normal bundle for $Z$ in $X$. Then we have a localization sequence for $\mathbf{W}$-cohomology
\[
\cdots\to \op{H}^i(U,\mathbf{W}(\mathscr{L}))\xrightarrow{\partial} \op{H}^{i-c+1}(Z,\mathbf{W}(\mathscr{L}\otimes \mathscr{N}_Z)) \xrightarrow{i_\ast}
\]
\[\xrightarrow{i_\ast}\op{H}^{i+1}(X,\mathbf{W}(\mathscr{L})) \xrightarrow{j^\ast} \op{H}^{i+1}(U,\mathbf{W}(\mathscr{L}))\to\cdots
\]
This has the distinct advantage that there are no index shifts in the coefficients (such as what happens for $\mathbf{I}$-cohomology) and we really get an honest long exact sequence (as opposed to only a piece of a long exact sequence containing the ``geometric bidegrees''). This way, computations of $\mathbf{W}$-cohomology can follow their classical topology counterparts much more closely than is possible for $\mathbf{I}$-cohomology.

\begin{remark}
  One explanation of the simplified form of the localization sequence is that the $\mathbf{W}$-cohomology ring $\bigoplus_n\op{H}^\bullet(X,\mathbf{W})$ considered above is part of the $\eta$-inverted Witt group theory considered e.g. in \cite{ananyevskiy}. Essentially, it is the quotient of the $\eta$-inverted Witt ring of $X$ modulo $\eta-1$. Some of the formulas for $\mathbf{W}$-cohomology of Grassmannians we develop in this paper already appear in loc.~cit.
\end{remark}

\section{Characteristic classes for vector bundles}
\label{sec:vbclasses}

The next two sections will provide a computation of the Chow--Witt ring of ${\op{B}}\op{GL}_n$. The global structure of the argument is similar to the computation of integral cohomology with local coefficients of ${\op{B}}\op{O}(n)$, cf.~\cite{cadek}. Some of the relevant adaptations to the Chow--Witt setting have already been made in \cite{chow-witt}. Additional, the decomposition of $\mathbf{I}$-cohomology into the image of $\beta$ and $\mathbf{W}$-cohomology will  significantly simplify the approach of \cite{chow-witt}, rendering the arguments even closer to their topological counterparts.

In this section, we begin by setting up the localization sequence and defining the relevant characteristic classes for vector bundles. We formulate the main structure results concerning the Chow--Witt and $\mathbf{I}^\bullet$-cohomology ring of ${\op{B}}\op{GL}_n$ and establish the basic relations between the characteristic classes. The inductive proof of the structure theorem will be done in the next section. 


\subsection{Setup of localization sequence}
We begin by setting up the localization sequence for the inductive computation of the cohomology of ${\op{B}}\op{GL}_n$, following the procedure for $\op{SL}_n$ in \cite[Section 5.1]{chow-witt}. 

Let $V$ be a finite-dimensional representation of $\op{GL}_n$ on which $\op{GL}_n$ acts freely outside a closed stable subset $Y$ of codimension $s$, and consider the quotient $X(V):=(V\setminus Y)/\op{GL}_n$. This computes $\widetilde{\op{CH}}^\bullet({\op{B}}\op{GL}_n,\mathscr{L})$ in degrees $\leq s-2$, and a finite-dimensional model for the universal $\op{GL}_n$-torsor is given by the projection $p\colon V\setminus Y\to X(V)$. The tautological $\op{GL}_n$-representation on $\mathbb{A}^n$ gives rise to a vector bundle $\gamma_V\colon E_n(V)\to X(V)$ associated to the $\op{GL}_n$-torsor $p\colon V\setminus Y\to X(V)$. 

Denote by $S_n(V)$ the complement of the zero-section of $\gamma_V\colon E(V)\to X(V)$. As in the case of $\op{SL}_n$, the complement $S_n(V)$ can be identified as an approximation of the classifying space ${\op{B}}\op{GL}_{n-1}$. Moreover, the quotient map $q\colon (V\setminus Y)/\op{GL}_{n-1}\to X(V)$ induces a morphism 
\[
\widetilde{\op{CH}}^\bullet(X(V),\mathscr{L})\xrightarrow{\gamma_V^\ast} \widetilde{\op{CH}}^\bullet(S_n(V),\gamma_V^\ast(\mathscr{L})) \cong \widetilde{\op{CH}}^\bullet((V\setminus Y)/\op{GL}_{n-1},q^\ast(\mathscr{L}))
\]
which models the stabilization map $\widetilde{\op{CH}}^\bullet({\op{B}}\op{GL}_n,\mathscr{L})\to \widetilde{\op{CH}}^\bullet({\op{B}}\op{GL}_{n-1},\iota^\ast\mathscr{L})$ for the standard inclusion $\iota\colon {\op{B}}\op{GL}_{n-1}\to{\op{B}}\op{GL}_n$. Consequently, we get the following localization sequence:

\begin{proposition}
\label{prop:locgln}
There is a long exact sequence of Chow--Witt groups of classifying spaces 
\begin{eqnarray*}
\cdots &\to & \widetilde{\op{CH}}^{q-n}({\op{B}}\op{GL}_n,\mathscr{L}\otimes\det\gamma_n) \to \widetilde{\op{CH}}^q({\op{B}}\op{GL}_n,\mathscr{L}) \to\\ &\to&\widetilde{\op{CH}}^q({\op{B}}\op{GL}_{n-1},\iota^\ast(\mathscr{L})) \to \op{H}^{q+1-n}({\op{B}}\op{GL}_n,\mathbf{K}^{\op{MW}}_{q-n}(\mathscr{L}\otimes\det\gamma_n)) \to \\&\to& \op{H}^{q+1}({\op{B}}\op{GL}_n,\mathbf{K}^{\op{MW}}_{q}(\mathscr{L}))\to \cdots
\end{eqnarray*}
The first map is the composition of the d\'evissage isomorphism with the forgetting of support, alternatively ``multiplication with the Euler class of the universal bundle $\gamma_n$''. The second map is the restriction along the stabilization inclusion $\iota\colon \op{GL}_{n-1}\to\op{GL}_n$. 

There are similar exact sequences for the other coefficients $\mathbf{I}^\bullet(\mathscr{L})$, $\mathbf{K}^{\op{M}}_\bullet$ and $\mathbf{W}$, and the change-of-coefficients maps induce commutative ladders of exact sequences. Notably, the localization sequence for $\mathbf{W}$-cohomology looks as follows:
\[
\cdots\to \op{H}^{q-n}({\op{B}}\op{GL}_n,\mathbf{W}(\mathscr{L}\otimes \det\gamma_n))\xrightarrow{\op{e}_n} \op{H}^q({\op{B}}\op{GL}_n,\mathbf{W}(\mathscr{L}))\xrightarrow{\iota^\ast}
\]
\[
\xrightarrow{\iota^\ast} \op{H}^q({\op{B}}\op{GL}_{n-1},\mathbf{W}(\iota^\ast\mathscr{L}))\xrightarrow{\partial} \op{H}^{q-n+1}({\op{B}}\op{GL}_n,\mathbf{W}(\mathscr{L}\otimes \det\gamma_n))\to\cdots
\]
\end{proposition}

The proof is the same line of argument as for the case $\op{SL}_n$ in \cite[Proposition 5.1]{chow-witt}. 
\begin{remark}
Note also that for $\mathscr{L}=\det\gamma_n$, with $\gamma_n$ the universal rank $n$ bundle on ${\op{B}}\op{GL}_n$, we have $\iota^\ast\mathscr{L}\cong \det\gamma_{n-1}$. The multiplication with the Euler class changes the dualities. 
\end{remark}

\subsection{Euler class}

Recall from \cite[Definition 5.9]{chow-witt} how the Chow--Witt-theoretic Euler class of \cite{AsokFaselEuler} gives rise to an Euler class in $\widetilde{\op{CH}}^\bullet({\op{B}}\op{GL}_n,\det \gamma_n^\vee)$. For a smooth scheme $X$, the Chow--Witt-theoretic Euler class of a vector bundle $p\colon \mathscr{E}\to X$ of rank $n$ is defined via the formula
\[
\op{e}_n(p\colon \mathscr{E}\to X):=(p^\ast)^{-1}{s_0}_\ast(1)\in\widetilde{\op{CH}}^n(X,\det(p)^\vee),
\]
where $s_0\colon X\to \mathscr{E}$ is the zero section. 
Using smooth finite-dimensional approximations to the classifying space ${\op{B}}\op{GL}_n$ provides a well-defined Euler class 
\[
\op{e}_n\in \widetilde{\op{CH}}^n({\op{B}}\op{GL}_n,\det(\gamma_n)^\vee).
\]
In the localization sequence of Proposition~\ref{prop:locgln}, the Euler class corresponds under the d\'evissage isomorphism to the Thom class for the universal rank $n$ vector bundle $\gamma_n$ on ${\op{B}}\op{GL}_n$. This 
justifies calling the composition
\[
\widetilde{\op{CH}}^{q-n}({\op{B}}\op{GL}_n,\mathscr{L}\otimes\det\gamma_n) \cong \widetilde{\op{CH}}^q_{{\op{B}}\op{GL}_n}(E_n,\mathscr{L}) \to  \widetilde{\op{CH}}^q(E_n,\mathscr{L}) \cong \widetilde{\op{CH}}^q({\op{B}}\op{GL}_n,\mathscr{L})
\]
``multiplication with the Euler class''. There are corresponding notions of Euler classes in $\mathbf{I}^\bullet$-cohomology, $\mathbf{W}$-cohomology, as well as Chow theory; these are compatible with the change of coefficients. The Euler classes are compatible with pullbacks of morphisms between smooth schemes, cf.~\cite[Proposition 3.1.1]{AsokFaselEuler}.

\subsection{Chern classes}

A direct consequence of the above localization sequence for Chow theory is the computation of the Chow-ring (with integral and mod 2 coefficients) of the classifying space ${\op{B}}\op{GL}_n$. The formulas are the standard ones found in any intersection theory handbook, cf. also~\cite[Proposition 5.2]{chow-witt}. As in loc.cit., the Chern classes are uniquely determined by their compatibility with stabilization and the identification of the top Chern class with the Euler class of the universal bundle. 

\begin{proposition}
\label{prop:chern}
There are unique classes $\op{c}_i(\op{GL}_n)\in \op{CH}^i({\op{B}}\op{GL}_n)$ for $1\leq i\leq n$, such that the natural stabilization morphism $\iota\colon {\op{B}}\op{GL}_{n-1}\to{\op{B}}\op{GL}_n$ satisfies $\iota^\ast \op{c}_i(\op{GL}_n)=\op{c}_i(\op{GL}_{n-1})$ for $i<n$ and $\op{c}_n(\op{GL}_n)=\op{e}_n(\op{GL}_n)$. In particular, the Chow--Witt-theoretic Euler class reduces to the top Chern class in the Chow theory. 
There is a natural isomorphism
\[
\op{CH}^\bullet({\op{B}}\op{GL}_n)\cong\mathbb{Z}[\op{c}_1,\op{c}_2,\dots,\op{c}_n]. 
\]
The restriction along the Whitney sum ${\op{B}}\op{GL}_m\times{\op{B}}\op{GL}_{n-m}\to{\op{B}}\op{GL}_n$ maps the Chern classes as follows:
\[
\op{c}_i\mapsto \sum_{j=i+m-n}^m\op{c}_j\boxtimes\op{c}_{i-j}.
\]
\end{proposition}

\begin{remark}
\label{rem:bglnorient}
From the above computations of the Chow ring of ${\op{B}}\op{GL}_n$ we also see the standard fact that $\op{Pic}({\op{B}}\op{GL}_n)\cong\mathbb{Z}$. Note that for any smooth scheme $X$ and any two line bundles $\mathscr{L},\mathscr{N}$ over $X$ such that the class of $\mathscr{L}$ in $\op{Pic}(X)$ is divisible by $2$, we have 
\[
\widetilde{\op{CH}}^\bullet(X,\mathscr{L}\otimes\mathscr{N})\cong \widetilde{\op{CH}}^\bullet(X,\mathscr{N}).
\]
In particular, there are only two relevant dualities to consider for ${\op{B}}\op{GL}_n$: the trivial duality corresponding to the trivial line bundle on ${\op{B}}\op{GL}_n$, and the nontrivial duality corresponding to the determinant of the universal bundle. This closely resembles the classical situation where $\pi_1{\op{B}}\op{O}(n)\cong\mathbb{Z}/2\mathbb{Z}$ and so there are only two isomorphism classes of local systems on ${\op{B}}\op{O}(n)$. 
\end{remark}

\subsection{Pontryagin classes}

Recall from \cite[Definition 5.6]{chow-witt} that the Pontryagin classes of vector bundles are defined as the images of $\op{p}_i\in \widetilde{\op{CH}}^\bullet({\op{B}}\op{Sp}_{2n})$ of \cite[Theorem 4.10]{chow-witt} under the homomorphism
\[
\widetilde{\op{CH}}^\bullet({\op{B}}\op{Sp}_{2n})\to \widetilde{\op{CH}}^\bullet({\op{B}}\op{GL}_n)
\]
which is induced from the symplectification morphism (aka standard hyperbolic functor) ${\op{B}}\op{GL}_n\to {\op{B}}\op{Sp}_{2n}$. Note that this means that the Pontryagin classes of vector bundles always live in the Chow--Witt ring with trivial duality (because they come from the symplectic group). As for the special linear groups, cf.~\cite[Proposition 5.8]{chow-witt}, the Pontryagin classes are compatible with stabilization in the sense that 
\[
\iota^\ast(\op{p}_i(\op{GL}_n))=\op{p}_i(\op{GL}_{n-1})
\]
where $i<n$ and $\iota^\ast\colon \widetilde{\op{CH}}^\bullet({\op{B}}\op{GL}_n)\to \widetilde{\op{CH}}^\bullet({\op{B}}\op{GL}_{n-1})$ is induced from the natural stabilization map $\op{GL}_{n-1}\to\op{GL}_n$. There are corresponding definitions of Pontryagin classes for $\mathbf{I}^\bullet$-cohomology and $\mathbf{W}$-cohomology, compatible with the natural change-of-coefficient maps
\[
\widetilde{\op{CH}}^q(X) \to \op{H}^q(X,\mathbf{I}^q)\to \op{H}^q(X,\mathbf{W}).
\]

\subsection{Stiefel--Whitney classes and their (twisted) Bocksteins}

The localization sequence of Proposition~\ref{prop:locgln} immediately implies a theory of Stiefel--Whitney classes which are determined by the compatibility with stabilization and the identification of the top Stiefel--Whitney class with the Euler class of the respective universal bundle, cf. \cite[Proposition 5.4]{chow-witt}. 

\begin{proposition}
\label{prop:sw}
There are unique classes $\overline{\op{c}}_i(\op{GL}_n)\in \op{Ch}^i({\op{B}}\op{GL}_n)$ for $1\leq i\leq n$, such that the natural stabilization morphism $\iota\colon {\op{B}}\op{GL}_{n-1}\to{\op{B}}\op{GL}_n$ satisfies $\iota^\ast \overline{\op{c}}_i(\op{GL}_n)=\overline{\op{c}}_i(\op{GL}_{n-1})$ for $i<n$ and $\overline{\op{c}}_n(\op{GL}_n)=\op{e}_n(\op{GL}_n)$. These agree with the Stiefel--Whitney classes in \cite[Definition 4.2]{fasel:ij}. There is a natural isomorphism
\[
\op{Ch}^\bullet({\op{B}}\op{GL}_n)\cong \mathbb{Z}/2\mathbb{Z}[\overline{\op{c}}_1,\dots,\overline{\op{c}}_n].
\]
\end{proposition}

Again, this is a very classical formula. We include it just for the following discussion of the (twisted) Bockstein classes and the action of the respective (twisted) Steenrod squares on $\op{Ch}^\bullet({\op{B}}\op{GL}_n)$. 

Recall from Section~\ref{sec:recall} that for a scheme $X$ and a line bundle $\mathscr{L}$, we have a Bockstein map $\op{Ch}^n(X)\to\op{H}^{n+1}(X,\mathbf{I}^{n+1}(\mathscr{L}))$. For the specific case of ${\op{B}}\op{GL}_n$, there are two relevant line bundles to consider: $\mathscr{O}$ and $\det\gamma_n^\vee$, cf.~Remark~\ref{rem:bglnorient}. This leads to two types of Bockstein classes for vector bundles:

\begin{definition}
\label{def:bockstein}
For a set $J=\{j_1,\dots,j_l\}$ of integers $0<j_1<\cdots<j_l\leq[(n-1)/2]$, there are classes
\begin{eqnarray*}
\beta_J&:=&\beta_{\mathscr{O}}(\overline{\op{c}}_{2j_1} \overline{\op{c}}_{2j_2}\cdots\overline{\op{c}}_{2j_l})\in \op{H}^{d+1}({\op{B}}\op{GL}_n,\mathbf{I}^{d+1})\\
\tau_J&:=&\beta_{\det\gamma_n^\vee}(\overline{\op{c}}_{2j_1} \overline{\op{c}}_{2j_2}\cdots\overline{\op{c}}_{2j_l})\in \op{H}^{d+1}({\op{B}}\op{GL}_n,\mathbf{I}^{d+1}(\det\gamma_n^\vee))
\end{eqnarray*}
where $d=\sum_{a=1}^l 2j_a$. 
\end{definition}

\begin{remark}
The Bockstein class $\beta(\emptyset)$ is trivial, cf.~\cite[Remark 5.12]{chow-witt}. However, the class $\tau(\emptyset)$ is nontrivial; more precisely, 
\[
\rho(\tau(\emptyset))=\op{Sq}^2_{\det\gamma_n^\vee}(1)=\overline{\op{c}}_1.
\]
\end{remark}

\begin{lemma}
\label{lem:baertorsion}
For a set $J=\{j_1,\dots,j_l\}$ of integers $0<j_1<\cdots<j_l\leq[(n-1)/2]$, we have  
\[
\op{I}(F)\beta_{\mathscr{O}}(\overline{\op{c}}_{2j_1}\cdots \overline{\op{c}}_{2j_l})=0 \quad \textrm{ and } \quad \op{I}(F)\beta_{\det\gamma_n^\vee}(\overline{\op{c}}_{2j_1} \cdots \overline{\op{c}}_{2j_l})=0
\]
in $\op{H}^\bullet({\op{B}}\op{GL}_n,\mathbf{I}^\bullet)$ and  $\op{H}^\bullet({\op{B}}\op{GL}_n,\mathbf{I}^\bullet(\det\gamma_n^\vee))$, respectively.
\end{lemma}

\begin{proof}
As in \cite[Lemma 7.3]{chow-witt}, this is formal from the $\op{W}(F)$-linearity of the maps in the exact B\"ar sequence.
\end{proof}

\begin{proposition}
\label{prop:eulerrel}
With the notation from Definition~\ref{def:bockstein}, if $n=2k+1$, we have 
\[
\op{e}_n=\beta_{\det\gamma_n^\vee}(\overline{\op{c}}_{n-1})=\tau_{\{k\}}.
\]
\end{proposition}

\begin{proof}
This is proved in \cite[Theorem 10.1]{fasel:ij}, noting that our Stiefel--Whitney classes in Proposition~\ref{prop:sw} agree with those in loc.cit., cf. also \cite[Proposition 7.5]{chow-witt}. 
\end{proof}

Combining Lemma~\ref{lem:baertorsion} and Proposition~\ref{prop:eulerrel}, we see that the Euler class $\op{e}_n\in\op{H}^n({\op{B}}\op{GL}_n,\mathbf{I}^n(\det\gamma_n^\vee))$ is $\op{I}(F)$-torsion if $n$ is odd.

\begin{remark}
  Note that on ${\op{B}}\op{GL}_n$, the Bockstein classes don't contain more information than the Stiefel--Whitney classes; it will follow from Proposition~\ref{prop:glnw} combined with Lemma~\ref{lem:wsplit}, the reduction morphism
  \[
  \rho\colon \op{H}^m({\op{B}}\op{GL}_n,\mathbf{I}^m(\mathscr{L}))\to \op{Ch}^m({\op{B}}\op{GL}_n)
  \]
  is injective on the image of $\beta_{\mathscr{L}}$. However, for a smooth scheme $X$, it is possible that the Bockstein class is nontrivial while its reduction in the mod 2 Chow ring is trivial. Topologically, this happens if the integral Stiefel--Whitney class is divisible by $2$; divisibility results for the integral Stiefel--Whitney classes arise e.g. in Massey's discussion of the obstruction theory for existence of almost complex structures. 
\end{remark}

\subsection{The Wu formula for the Chow-ring}

We shortly discuss the action of the Steenrod squares $\op{Sq}^2_{\mathscr{L}}$ on $\op{Ch}^\bullet({\op{B}}\op{GL}_n)$. Essentially, this is described by the Wu formula. 

\begin{proposition}
\label{prop:wuformula}
The untwisted Steenrod square $\op{Sq}^2_{\mathscr{O}}$ is given by 
\[
\op{Sq}^2_{\mathscr{O}}\colon  \op{Ch}^\bullet({\op{B}}\op{GL}_n) \to \op{Ch}^\bullet({\op{B}}\op{GL}_n)\colon \overline{\op{c}}_j\mapsto \overline{\op{c}}_1\overline{\op{c}}_j+(j-1)\overline{\op{c}}_{j+1}.
\]
The twisted Steenrod square $\op{Sq}^2_{\det\gamma_n}$ is given by 
\[
\op{Sq}^2_{\det\gamma_n}\colon  \op{Ch}^\bullet({\op{B}}\op{GL}_n) \to \op{Ch}^\bullet({\op{B}}\op{GL}_n)\colon \overline{\op{c}}_j\mapsto (j-1)\overline{\op{c}}_{j+1}.
\]
(Twisted) Steenrod squares of other elements are determined by the above formulas, the derivation property of the Steenrod square $\op{Sq}^2_{\mathscr{O}}$ and the relation $\op{Sq}^2_{\delta\gamma_n}(x)=\overline{\op{c}}_1\cdot x+\op{Sq}^2_{\mathscr{O}}(x)$. 
\end{proposition}

\begin{proof}
  Probably the Wu formula for the Stiefel--Whitney classes in Chow theory mod 2 is well-known, but right now I don't know of a reference. The first and second statement are equivalent by using Proposition~\ref{prop:twistedsq2} and noting that $\overline{\op{c}}_1(\det\gamma_n)=\overline{\op{c}}_1$. The second statement for odd Stiefel--Whitney classes is proved in \cite[Proposition 10.3]{fasel:ij}. For even Stiefel--Whitney classes, the required vanishing follows from
  \begin{eqnarray*}
    \op{Sq}^2_{\det\gamma_n}\circ \op{Sq}^2_{\det\gamma_n}(x)&=&\overline{\op{c}}_1\cdot \op{Sq}^2_{\det\gamma_n}(x)+\op{Sq}^2_{\mathscr{O}}\circ\op{Sq}^2_{\det\gamma_n}(x)\\
    &=&\overline{\op{c}}_1^2\cdot x+\overline{\op{c}}_1\cdot\op{Sq}^2_{\mathscr{O}}(x) + \op{Sq}^2_{\mathscr{O}}(\overline{\op{c}}_1\cdot x)+\op{Sq}^2_{\mathscr{O}}\circ\op{Sq}^2_{\mathscr{O}}(x)\\
    &=&\overline{\op{c}}_1^2\cdot x+\overline{\op{c}}_1\cdot\op{Sq}^2_{\mathscr{O}}(x) + \overline{\op{c}}_1\cdot \op{Sq}^2_{\mathscr{O}}(x)+x\cdot\op{Sq}^2_{\mathscr{O}}(\overline{\op{c}}_1)=0
    \end{eqnarray*}
\end{proof}

\begin{corollary}
\label{cor:wukernel}
The kernel of the untwisted Steenrod square $\op{Sq}^2_{\mathscr{O}}$ is given by the subring 
\[
\mathbb{Z}/2\mathbb{Z}[\overline{\op{c}}_i^2,\overline{\op{c}}_1\overline{\op{c}}_{2i}+\overline{\op{c}}_{2i+1},\overline{\op{c}}_1\overline{\op{c}}_n] \subseteq\mathbb{Z}/2\mathbb{Z}[\overline{\op{c}}_1,\dots,\overline{\op{c}}_n]= \op{Ch}^\bullet({\op{B}}\op{GL}_n).
\]
The kernel of the twisted Steenrod square $\op{Sq}^2_{\det\gamma_n}$ is given by the subring 
\[
\mathbb{Z}/2\mathbb{Z}[\overline{\op{c}}_{2i+1},\overline{\op{c}}_{2i}^2,\overline{\op{c}}_n] \subseteq\mathbb{Z}/2\mathbb{Z}[\overline{\op{c}}_1,\dots,\overline{\op{c}}_n]= \op{Ch}^\bullet({\op{B}}\op{GL}_n).
\]
\end{corollary}

\begin{proof}
The claims follow from the Wu formula in Proposition~\ref{prop:wuformula}. The twisted Steenrod square is given essentially by the same formula as the Steenrod square in $\op{Ch}^\bullet({\op{B}}\op{SL}_n)$, hence the formulas from \cite{chow-witt} apply. For the untwisted Steenrod square $\op{Sq}^2_{\mathscr{O}}$, the even classes $\overline{\op{c}}_{2i}$ map to $\overline{\op{c}}_1\overline{\op{c}}_{2i}+\overline{\op{c}}_{2i+1}$, hence the latter classes are in the kernel of the Steenrod square; similarly for $\overline{\op{c}}_1\overline{\op{c}}_n$. The description of the kernels follow from that, cf. also \cite[p. 285]{cadek}. 
\end{proof}

\begin{corollary}
  \label{cor:wuimage}
  Consider the mod 2 Chow ring $\op{Ch}^\bullet({\op{B}}\op{GL}_n)$. The images of the Steenrod squares $\op{Sq}^2_{\mathscr{L}}$ are contained in the subring generated by $\op{Sq}^2_{\mathscr{O}}(\overline{\op{c}}_{2j_1}\cdots\overline{\op{c}}_{2j_l})$,  $\op{Sq}^2_{\det\gamma_n}(\overline{\op{c}}_{2j_1}\cdots\overline{\op{c}}_{2j_l})$, $\op{Sq}^2_{\det\gamma_n}(1)$, $\overline{\op{c}}_{2i}^2$ and $\overline{\op{c}}_n$. 
\end{corollary}

\begin{proof}
  The Steenrod squares $\op{Sq}^2_{\mathscr{L}}$ are linear. To determine generators for the image, it thus suffices to consider Steenrod squares of monomials in the Chern classes.

  Since $\op{Sq}^2_{\mathscr{O}}$ is a derivation, we have $\op{Sq}^2_{\mathscr{O}}(x^2)=2x\op{Sq}^2_{\mathscr{O}}(x)=0$ and $\op{Sq}^2_{\mathscr{O}}(x^2y)=x^2\op{Sq}^2_{\mathscr{O}}(y)$. In particular, we can always pull out squares. For even Stiefel--Whitney classes, these squares are explicitly included as generators in the statement. For the odd Stiefel--Whitney classes, we have
  \[
  \op{Sq}^2_{\mathscr{O}}(\overline{\op{c}}_{2i}\overline{\op{c}}_{2i+1})=\overline{\op{c}}_{2i}\op{Sq}^2_{\mathscr{O}}(\overline{\op{c}}_{2i+1})+\overline{\op{c}}_{2i+1}\op{Sq}^2_{\mathscr{O}}(\overline{\op{c}}_{2i})=2\overline{\op{c}}_1\overline{\op{c}}_{2i}\overline{\op{c}}_{2i+1}+\overline{\op{c}}_{2i+1}^2.
  \]
  It thus suffices to show that the Steenrod squares of all products $\overline{\op{c}}_{j_1}\dots\overline{\op{c}}_{j_m}$ with no repeating factors are contained in the subring as claimed.

  For the odd Stiefel--Whitney classes we have
  \[
  \op{Sq}^2_{\mathscr{O}}(\overline{\op{c}}_{2i+1}x)=\overline{\op{c}}_{2i+1}\op{Sq}^2(x)+\overline{\op{c}}_1\overline{\op{c}}_{2i+1}x=\overline{\op{c}}_{2i+1}\op{Sq}^2_{\det\gamma_n}(x).
  \]
  Since $\overline{\op{c}}_{2i+1}=\op{Sq}^2_{\det\gamma_n}(\overline{\op{c}}_{2i})$ with the special case $\overline{\op{c}}_1=\op{Sq}^2_{\det\gamma_n}(1)$, the odd Stiefel--Whitney classes are also among the generators of the subring listed in the claim. Therefore,  we can also pull out all the odd Stiefel--Whitney classes from the products $\overline{\op{c}}_{j_1}\dots\overline{\op{c}}_{j_m}$. A similar calculation shows that we can also pull out $\overline{\op{c}}_n$, which is also included explicitly among the generators. We have thus established the claim for $\op{Sq}^2_{\mathscr{O}}$.

  To show the claim for $\op{Sq}^2_{\det\gamma_n}$, we first have $\op{Sq}^2_{\det\gamma_n}(x^2)=\overline{\op{c}}_1x^2$ and
  \[
  \op{Sq}^2_{\det\gamma_n}(x^2y)=\overline{\op{c}}_1x^2y+x^2\op{Sq}^2_{\mathscr{O}}(y)=x^2\op{Sq}^2_{\det\gamma_n}(y).
  \]
  This tells us again that we can always pull out squares. For the odd Stiefel--Whitney classes we have
  \[
  \op{Sq}^2_{\det\gamma_n}(\overline{\op{c}}_{2i+1}x)=\overline{\op{c}}_1\overline{\op{c}}_{2i+1}x+\overline{\op{c}}_{2i+1}\op{Sq}^2_{\det\gamma_n}(x)=\overline{\op{c}}_{2i+1}\op{Sq}^2_{\mathscr{O}}(x).
  \]
  Therefore,we can also pull out odd Stiefel--Whitney classes (and by a similar computation also $\overline{\op{c}}_n$). The claim is proved.
\end{proof}

\subsection{The candidate presentation}

We define an appropriate graded ring $\mathscr{R}_n/\mathscr{I}_n$ which we will prove to be isomorphic to  $\op{H}^\bullet({\op{B}}\op{GL}_n,\mathbf{I}^\bullet\oplus \mathbf{I}^\bullet(\det\gamma_n^\vee))$. The ring will be graded by $\mathbb{Z}\oplus\mathbb{Z}/2\mathbb{Z}$, where the degrees $(n,0)$ are those with $\mathbf{I}^\bullet$-coefficients, and the degrees $(n,1)$ are those with $\mathbf{I}^\bullet(\det\gamma_n^\vee)$-coefficients. Following \cite{cadek}, we use the notation $\Delta(J,J')=(J\cup J')\setminus(J\cap J')$ for the symmetric difference of two subsets $J$ and $J'$ of a given set. 

\begin{definition}
\label{def:rngln}
Let $F$ be a field of characteristic $\neq 2$ and denote by $\op{W}(F)$ the Witt ring of quadratic forms over $F$. For a natural number $n\geq 1$, we define the $\mathbb{Z}\oplus\mathbb{Z}/2\mathbb{Z}$-\emph{graded-commutative} $\op{W}(F)$-algebra 
\[
\mathscr{R}_n=\op{W}(F)\left[P_1,\dots,P_{[(n-1)/2]},X_n, B_J, T_J, T_\emptyset\right]. 
\]

The classes $P_i$ sit in degree $(4i,0)$, the class $X_n$ in degree $(n,1)$. For the classes $B_J$ and $T_J$, the index set $J$ runs through the sets $\{j_1,\dots,j_l\}$ of natural numbers with $0<j_1<\dots<j_l\leq[(n-1)/2]$, and the degrees of $B_J$ and $T_J$ are $(d,0)$ and $(d,1)$ with $d=1+2\sum_{a=1}^lj_a$, respectively. By convention $B_\emptyset=0$. 

Let $\mathscr{I}_n\subset \mathscr{R}_n$ be the ideal generated by the following relations:
\begin{enumerate}
\item $\op{I}(F) B_J=\op{I}(F)T_J=\op{I}(F)T_\emptyset=0$.
\item If $n=2k+1$ is odd, $X_{2k+1}=T_{\{k\}}$. 
\item For two index sets $J$ and $J'$ where $J'$ can be empty, we have 
\begin{eqnarray*}
B_J\cdot B_{J'}&=&\sum_{k\in J} B_{\{k\}}\cdot P_{(J\setminus\{k\})\cap J'}\cdot B_{\Delta(J\setminus\{k\},J')}\\
B_J\cdot T_{J'}&=&\sum_{k\in J}B_{\{k\}}\cdot P_{(J\setminus\{k\})\cap J'}\cdot T_{\Delta(J\setminus\{k\},J')}\\
T_J\cdot B_{J'}&=&B_J\cdot T_{J'} + T_\emptyset \cdot P_{J\cap J'}\cdot B_{\Delta(J,J')}\\
T_J\cdot T_{J'}&=&B_J\cdot B_{J'}+T_\emptyset\cdot P_{J\cap J'}\cdot T_{\Delta(J,J')}
\end{eqnarray*}
where we set $P_A=\prod_{i=1}^l P_{a_i}$ for an index set $A=\{a_1,\dots,a_l\}$.
\end{enumerate}
\end{definition}

\begin{remark}
  Note that there are slight differences in the indexing sets between the formulas in \cite{brown} and \cite{cadek}. For the Pontryagin classes, this difference is due to fact that the Euler class squares to the top Pontryagin class. So in \v Cadek's presentation, there is no need to introduce the top Pontryagin class; on the other hand, Brown only computes cohomology with trivial coefficients and he has to introduce the top Pontryagin class separately. The same thing is true for the Bockstein classes:
  \[
  \beta_{\mathscr{O}}(\overline{\op{c}}_{2j_1}\cdots\overline{\op{c}}_{2j_l})=\beta_{\det\gamma_n^\vee}(\overline{\op{c}}_{2j_1}\cdots\overline{\op{c}}_{2j_{l-1}})\op{e}_{n-1} \textrm{ if } j_l=(n-1)/2;
  \]
  and this relation cannot be expressed in cohomology with trivial coefficients. Moreover, the reason why Brown's additional $\overline{\op{c}}_1$-factors in the Bockstein classes can be omitted in \v Cadek's presentation is given by the formula
  \[
  \beta_{\mathscr{O}}(\overline{\op{c}}_1\overline{\op{c}}_{2j_1}\cdots\overline{\op{c}}_{2j_l})=\beta_{\det\gamma_n^\vee}(\overline{\op{c}}_{2j_1}\cdots\overline{\op{c}}_{2j_{l-1}})\beta_{\det\gamma_n^\vee}(1).
  \]
\end{remark}

\begin{definition}
\label{def:rnglnres}
Let $n\geq 2$ be a natural number. Define the $\op{W}(F)$-algebra homomorphism $\Phi_n\colon \mathscr{R}_n\to\mathscr{R}_{n-1}$ by setting
\begin{enumerate}
\item the element $P_i$ maps to $P_i$ if $i<(n-1)/2$ and maps to $X_{n-1}^2$ if $i=(n-1)/2$, 
\item the element $X_n$ maps to $0$, 
\item for the index set $J=\{j_1,\dots,j_l\}$, we have 
\[
B_J\mapsto \left\{\begin{array}{ll} B_J & j_l<(n-1)/2\\ T_{J'}\cdot X_{n-1} & j_l=(n-1)/2, J=J'\sqcup\{j_l\}\end{array}\right.
\]
\[
T_J\mapsto \left\{\begin{array}{ll} T_J & j_l<(n-1)/2\\ B_{J'}\cdot X_{n-1} & j_l=(n-1)/2, J=J'\sqcup\{j_l\}\end{array}\right.
\]
\end{enumerate}
\end{definition}

\begin{remark}
\label{rem:compatres}
The above formulas model the restriction of classes from ${\op{B}}\op{GL}_n$ to ${\op{B}}\op{GL}_{n-1}$. On the level of mod 2 Chow rings, we have  
\begin{eqnarray*}
\op{Sq}^2_{\mathscr{L}}(\overline{\op{c}}_{2j_1}\cdots \overline{\op{c}}_{2j_l})&=&\op{Sq}^2_{\mathscr{L}}(\overline{\op{c}}_{2j_1}\cdots \overline{\op{c}}_{2j_{l-1}})\overline{\op{c}}_{2j_l} + \overline{\op{c}}_1\overline{\op{c}}_{2j_1}\cdots\overline{\op{c}}_{2j_l}\\ &=&\op{Sq}^2_{\mathscr{L}\otimes\det\gamma_n^\vee}(\overline{\op{c}}_{2j_1}\cdots \overline{\op{c}}_{2j_{l-1}})\op{e}_{n-1},
\end{eqnarray*}
using Proposition~\ref{prop:twistedsq2}. Note that the formulas for restriction on the bottom of page~283 in \cite{cadek} contain some typos, the classes live in the wrong degrees.
\end{remark}

\begin{proposition}
\label{prop:compatres}
With the notation from Definitions~\ref{def:rngln} and \ref{def:rnglnres}, we have 
\[
\Phi_n(\mathscr{I}_n)\subseteq \mathscr{I}_{n-1}.
\]
In particular, the map $\Phi_n$ descends to a well-defined ring homomorphism 
\[
\overline{\Phi}_n\colon \mathscr{R}_n/\mathscr{I}_n\to \mathscr{R}_{n-1}/\mathscr{I}_{n-1}.
\]
\end{proposition}

\begin{proof}
We first deal with the relations of type (1). Recall that the map $\Phi_n$ is by definition $\op{W}(F)$-linear, in particular, it will send $\op{I}(F)$ to $\op{I}(F)$. Since $\Phi_n$ sends $B_J$ to either $B_J$ or $T_{J'}\cdot X_{n-1}$ (and similarly $T_J$ to either $T_J$ or $B_{J'}\cdot X_{n-1}$, with the special case $B_\emptyset=0$) it is clear the relations of type (1) are preserved. 

The relations of type (2) are also preserved since both $X_{2k+1}$ and $T_k$ are mapped to $0$ by $\Phi_n$. 

It remains to deal with relations of type (3). These relations are trivially preserved if neither $J$ nor $J'$ contains the highest possible index $j_l=(n-1)/2$. In this case, all the relevant $B_J$, $T_J$ and $P_J$ will exist both in $\mathscr{R}_n$ and $\mathscr{R}_{n-1}$, and the corresponding relation in $\mathscr{R}_n$ is just mapped to the same relation in $\mathscr{R}_{n-1}$. 

For relations of type (3.1), assume that $j_l\in J'$ and $j_l\not\in J$. On the left-hand side, $B_{J'}$ restricts to $T_{J'\setminus\{j_l\}}\cdot X_{n-1}$ and on the right-hand side, $B_{\Delta(J\setminus\{k\},J')}$ restricts to $B_{\Delta(J\setminus\{k\},J'\setminus\{j_l\})}\cdot X_{n-1}$. The result is the product of a relation of type (3.2) with $X_{n-1}$. Conversely, if $j_l\in J$ and $j_l\not\in J'$, then the left-hand side restricts to $T_{J\setminus\{j_l\}}\cdot X_{n-1}\cdot B_{J'}$. The right-hand side restricts to 
\[
\sum_{k\in J\setminus\{ j_l\}} B_{\{k\}}\cdot P_{(J\setminus \{k\})\cap J'} \cdot T_{\Delta(J\setminus\{k,j_l\},J')}\cdot X_{n-1} + T_\emptyset \cdot X_{n-1}\cdot P_{(J\setminus \{j_l\})\cap J'}\cdot B_{\Delta(J\setminus\{j_l\},J')}.
\]
But this is the product of a relation of type (3.3) and $X_{n-1}$. Finally, the case where $j_l\in J\cap J'$, the left-hand side restricts to $T_{J\setminus\{j_l\}}\cdot T_{J'\setminus\{j_l\}}\cdot X_{n-1}^2$. The right-hand side restricts to 
\[
\sum_{k\in J\setminus\{ j_l\}} B_{\{k\}}\cdot P_{(J\setminus \{k,j_l\})\cap J'}\cdot X_{n-1}^2 \cdot B_{\Delta(J\setminus\{k\},J')} + T_\emptyset\cdot P_{(J\setminus \{j_l\})\cap J'}\cdot T_{\Delta(J,J')}\cdot X_{n-1}^2.
\]
This is a product of a relation of type (3.4) with $X_{n-1}^2$. The argument for restriction of relations of type (3.2) is completely analogous. 

For the restriction of relations of type (3.4), if $j_l\in J$ and $j_l\not\in J'$, the left-hand side restricts to $B_J\cdot T_{J'}\cdot X_{n-1}$. The right-hand side restricts to $T_J\cdot B_{J'}\cdot X_{n-1}+T_\emptyset\cdot P_{J\cap J'}\cdot B_{\Delta(J\setminus \{j_l\},J')}\cdot X_{n-1}$. This is the product of a relation of type (3.3) with $X_{n-1}$, noting that all terms here are 2-torsion. All the other cases are done similarly.

Since $\Phi_n(\mathscr{I}_n)\subset \mathscr{I}_{n-1}$, it follows that the restriction map descends to a $\op{W}(F)$-algebra map $\overline{\Phi}_n\colon \mathscr{R}_n/\mathscr{I}_n\to \mathscr{R}_{n-1}/\mathscr{I}_{n-1}$ as claimed. 
\end{proof}

\begin{lemma}
\label{lem:evenpoly}
If $n$ is even, then we have an isomorphism 
\[
\mathscr{R}_n/\mathscr{I}_n\cong\mathscr{R}_{n-1}/\mathscr{I}_{n-1}[X_n].
\] 
In particular, the restriction map $\overline{\Phi}_n\colon \mathscr{R}_n/\mathscr{I}_n\to\mathscr{R}_{n-1}/\mathscr{I}_{n-1}$ is surjective. 
\end{lemma}

\begin{proof}
The index sets for the elements $P_i$ are the same for $n$ and $n-1$. In particular, $i\neq (n-1)/2$ which means that the $P_i$ in $\mathscr{R}_n$ are just mapped to the $P_i$ in $\mathscr{R}_{n-1}$. The same is true for the index sets for $B_J$ and $T_J$. Moreover, in $\mathscr{R}_{n-1}/\mathscr{I}_{n-1}$ we have $X_{n-1}=T_{\{(n-2)/2\}}$. This proves the surjectivity of $\Phi_n$. The claim about the polynomial ring follows since $X_n$ doesn't appear in any relation in $\mathscr{R}_n$. 
\end{proof}

\begin{lemma}
\label{lem:oddproblem}
If $n$ is odd, then there is an exact sequence of graded $\op{W}(F)$-algebras:
\[
\mathscr{R}_n/\mathscr{I}_n\xrightarrow{\overline{\Phi_n}} \mathscr{R}_{n-1}/\mathscr{I}_{n-1} \to \op{W}(F)[X_{n-1}]/(X_{n-1}^2)\to 0.
\] 
\end{lemma}

\begin{proof}
The elements $P_i\in\mathscr{R}_n$ with $i<(n-1)/2$ are mapped under ${\Phi_n}$ to the elements with the same name in $\mathscr{R}_{n-1}$. The same holds for the elements $B_J$ and $T_J$ where the index set $J$ doesn't contain $(n-1)/2$. In particular, the subalgebra of $\mathscr{R}_{n-1}/\mathscr{I}_{n-1}$ generated by all $P_i$, $B_J$ and $T_J$ is in the image. The only elements in $\mathscr{R}_n$ we have not yet considered so far are the new $P_{(n-1)/2}$ and the elements $B_J$ and $T_J$ where $J$ contains $(n-1)/2$.  The element $X_{n-1}^2$ is in the image of $P_{(n-1)/2}$, the elements $B_J'X_{n-1}$ are in the image of $T_J$ and the elements $T_J'X_{n-1}$ are in the image of $B_J$. However, the element $X_{n-1}$ itself is not in the image since we noted in Lemma~\ref{lem:evenpoly} that it is a polynomial variable in $\mathscr{R}_{n-1}$. Consequently, defining the morphism $\mathscr{R}_{n-1}/\mathscr{I}_{n-1}\to \op{W}(F)[X_{n-1}]/(X_{n-1}^2)$ by sending $X_{n-1}$ to itself and all the other generators to $0$ yields the desired exact sequence.
\end{proof}

\subsection{Statement of results}

Now we are ready to state the main theorem describing the $\mathbf{I}^\bullet$-cohomology and Chow--Witt ring of ${\op{B}}\op{GL}_n$. For the $\mathbf{I}^\bullet$-cohomology, the result is very close to \v Cadek's computation of the integral cohomology of ${\op{B}}\op{O}(n)$ with twisted coefficients, cf. \cite{cadek}. 

\begin{theorem}
\label{thm:glnin}
Let $n\geq 1$ be a natural number. 
\begin{enumerate}
\item The following ring homomorphism
\begin{eqnarray*}
\theta_n\colon \mathscr{R}_n&\to&  \op{H}^\bullet({\op{B}}\op{GL}_n,\mathbf{I}^\bullet\oplus\mathbf{I}^\bullet(\det\gamma_n^\vee))\colon \\ P_i&\mapsto&\op{p}_{2i} \\X_n&\mapsto& \op{e}_n \\
B_J&\mapsto& \beta_{\mathscr{O}}(\overline{\op{c}}_{2j_1}\cdots\overline{\op{c}}_{2j_l}) \textrm{ for } J=\{j_1,\dots,j_l\} \\ T_J&\mapsto& \beta_{\det\gamma_n^\vee}(\overline{\op{c}}_{2j_1}\cdots\overline{\op{c}}_{2j_l}) \textrm{ for } J=\{j_1,\dots,j_l\}\\ T_\emptyset&\mapsto& \beta_{\det\gamma_n^\vee}(1)
\end{eqnarray*}
induces a ring isomorphism $\overline{\theta}_n\colon \mathscr{R}_n/\mathscr{I}_n \xrightarrow{\cong}\op{H}_{\op{Nis}}^\bullet({\op{B}}\op{GL}_n,\mathbf{I}^\bullet\oplus\mathbf{I}^\bullet(\det\gamma_n^\vee))$. 
\item 
For any line bundle $\mathscr{L}$ on ${\op{B}}\op{GL}_n$, the reduction morphism 
\[
\op{H}^{\bullet}({\op{B}}\op{GL}_n,\mathbf{I}^\bullet(\mathscr{L}))\to \op{Ch}^\bullet({\op{B}}\op{GL}_n)
\]
induced from the projection $\mathbf{I}^n(\mathscr{L}) \mapsto\mathbf{K}^{\op{M}}_n/2$ is explicitly given by mapping
\[
\op{p}_{2i}\mapsto \overline{\op{c}}_{2i}^2, \; \beta_{\mathscr{L}}(\overline{\op{c}}_{2j_1}\cdots\overline{\op{c}}_{2j_l})\mapsto \op{Sq}^2_{\mathscr{L}}(\overline{\op{c}}_{2j_1}\cdots\overline{\op{c}}_{2j_l}),\; \op{e}_n\mapsto \overline{\op{c}}_n. 
\]
\item Any class $x$ in the ideal of $\op{H}^\bullet({\op{B}}\op{GL}_n,\mathbf{I}^\bullet\oplus \mathbf{I}^\bullet(\det\gamma_n^\vee))$ generated by $\beta_J$ and $\tau_J$ is trivial if and only if its reduction $\rho(x)\in\op{Ch}^\bullet({\op{B}}\op{GL}_n)$ is trivial. 
\end{enumerate}
\end{theorem}

The proof will be given in Section~\ref{sec:inproofs}. For now we draw some consequences concerning the structure of the Chow--Witt ring of ${\op{B}}\op{GL}_n$. 

\begin{proposition}
\label{prop:kerpartial}
\begin{enumerate}
\item 
The kernel of the composition 
\[
\partial_{\mathscr{O}}\colon \op{CH}^\bullet({\op{B}}\op{GL}_n)\to \op{Ch}^\bullet({\op{B}}\op{GL}_n)\xrightarrow{\beta_{\mathscr{O}}} \op{H}^\bullet({\op{B}}\op{GL}_n,\mathbf{I}^\bullet)
\]
 is the subring 
\begin{eqnarray*}
\ker\partial_{\mathscr{O}}&=&\mathbb{Z}[\op{c}_i^2, 2{\op{c}}_{i_1}\dots{\op{c}}_{i_k}, {\op{c}_1}{\op{c}_{2i}}+{\op{c}_{2i+1}},{\op{c}_1}{\op{c}_n}]\\ &\subseteq&\mathbb{Z}[\op{c}_1,\dots,\op{c}_n]\cong \op{CH}^\bullet({\op{B}}\op{GL}_n).
\end{eqnarray*}
\item 
The kernel of the composition 
\[
\partial_{\det\gamma_n^\vee}\colon \op{CH}^\bullet({\op{B}}\op{GL}_n)\to \op{Ch}^\bullet({\op{B}}\op{GL}_n)\xrightarrow{\beta_{\det\gamma_n^\vee}} \op{H}^\bullet({\op{B}}\op{GL}_n,\mathbf{I}^\bullet(\det\gamma_n^\vee))
\]
is the subring 
\begin{eqnarray*}
\ker\partial_{\det\gamma_n^\vee}&=&\mathbb{Z}[\op{c}_{2i+1}, 2{\op{c}}_{2i_1}\cdots {\op{c}}_{2i_k}, {\op{c}}_{2i}^2, \op{c}_n]\\&\subseteq&\mathbb{Z}[\op{c}_1,\dots,\op{c}_n]\cong \op{CH}^\bullet({\op{B}}\op{GL}_n).
\end{eqnarray*}
\end{enumerate}
\end{proposition}

\begin{proof}
The integral statements follow directly from the mod 2 statements: by (3) of Theorem~\ref{thm:glnin} and Proposition~\ref{prop:twistedsq2}, the kernel of $\beta_{\mathscr{L}}$ equals the kernel of $\op{Sq}^2_{\mathscr{L}}$ and the latter is determined by the Wu formula, cf. Corollary~\ref{cor:wukernel}. 
\end{proof}

\begin{theorem}
\label{thm:glnchw}
There is a cartesian square of $\mathbb{Z}\oplus\mathbb{Z}/2\mathbb{Z}$-graded $\op{GW}(F)$-algebras
\[
\xymatrix{
\widetilde{\op{CH}}^\bullet({\op{B}}\op{GL}_n,\mathscr{O}\oplus\det\gamma_n^\vee)\ar[r] \ar[d] & \ker\partial_{\mathscr{O}}\oplus \ker\partial_{\det\gamma_n^\vee} \ar[d]^{\bmod 2} \\
\op{H}^{\bullet}({\op{B}}\op{GL}_n,\mathbf{I}^\bullet\oplus\mathbf{I}^\bullet(\det\gamma_n^\vee))\ar[r]_>>>>>>\rho &\op{Ch}^\bullet({\op{B}}\op{GL}_n)^{\oplus 2}.
}
\]
The right vertical morphism is the natural reduction mod $2$ restricted to the kernels of the two boundary maps, and the lower horizontal morphism is the reduction morphism described in Theorem~\ref{thm:glnin}. The Chow--Witt-theoretic Euler class satisfies $\op{e}_n=(\op{e}_n,\op{c}_n)$ with $\op{c}_n\in\ker\partial_{\det\gamma_n^\vee}$. For the Chow--Witt-theoretic Pontryagin classes, we have
\[
\op{p}_{i}=\left(\op{p}_{i}, (-1)^i\op{c}_{i}^2+2\sum_{j=\max\{0,2i-n\}}^{i-1}(-1)^j\op{c}_j\op{c}_{2i-j}\right)
\]
where the odd Pontryagin classes in $\mathbf{I}$-cohomology are $\op{I}(F)$-torsion and satisfy $\op{p}_{2i+1}=\beta_{\mathscr{O}}(\overline{\op{c}}_{2i}\overline{\op{c}}_{2i+1})$. 

The top Pontryagin class $\op{p}_n\in\widetilde{\op{CH}}^{2n}({\op{B}}\op{Sp}_{2n})$ maps to $\op{e}_n^2\in \widetilde{\op{CH}}^{2n}({\op{B}}\op{GL}_n,\mathscr{O})$. 
\end{theorem}

\begin{proof}
The statement about the cartesian square follows directly from \cite[Proposition 2.11]{chow-witt}. The claims about the reduction from the Chow--Witt ring to $\mathbf{I}$-cohomology follows from the definition of the characteristic classes. The statement about $\op{e}_n$ and $\op{c}_n$ follows from Proposition~\ref{prop:chern}. The statement about the $\op{p}_i$ has been proved in \cite{chow-witt}. The statement about the top Pontryagin class is proved in Proposition~\ref{prop:glnw}, or \cite[Proposition 7.9]{chow-witt}. 
\end{proof}

\begin{example}
To clarify the relation between the cohomology of ${\op{B}}\op{GL}_n$ and ${\op{B}}\op{SL}_n$, cf. \cite[Example 6.12]{chow-witt}, we describe in detail the cartesian square for ${\op{B}}\op{GL}_3$ with both dualities. For the trivial duality, we have the following cartesian square: 
\[
\xymatrix{
\widetilde{\op{CH}}^\bullet({\op{B}}\op{GL}_3,\mathscr{O})\ar[r] \ar[d] & 2\mathbb{Z}[\op{c}_i^2,{\op{c}_1}{\op{c}_2}+{\op{c}_3},{\op{c}_1}{\op{c}_3}] \ar[d] \\
\op{W}(F)[\op{p}_2,\beta_{\mathscr{O}}(\overline{\op{c}}_1),\beta_{\mathscr{O}}(\overline{\op{c}}_2),\beta_{\mathscr{O}}(\overline{\op{c}}_1\overline{\op{c}}_2)]/\mathscr{I}_{3,\mathscr{O}} \ar[r] &\mathbb{Z}/2\mathbb{Z}[\overline{\op{c}}_1,\overline{\op{c}}_2,\overline{\op{c}}_3].
}
\]
For the nontrivial duality, we have the following square:
\[
\xymatrix{
\widetilde{\op{CH}}^\bullet({\op{B}}\op{GL}_3,\det\gamma_n^\vee)\ar[r] \ar[d] & \mathbb{Z}[\op{c}_1,2{\op{c}_2},\op{c}_2^2,\op{c}_3] \ar[d] \\
\op{W}(F)[\beta_{\det\gamma_n^\vee}(1),\beta_{\det\gamma_n^\vee}(\overline{\op{c}}_2)]/\mathscr{I}_{3,\det\gamma_n^\vee} \ar[r] &\mathbb{Z}/2\mathbb{Z}[\overline{\op{c}}_1,\overline{\op{c}}_2,\overline{\op{c}}_3].
}
\]
Note that the Euler class $\op{e}_3=\beta_{\det\gamma_n^\vee}(\overline{\op{c}}_2)$ lives in the cohomology with twisted coefficients. Note also that $\beta_{\mathscr{O}}(\overline{\op{c}}_1)=\beta_{\det\gamma_n^\vee}(1)^2$ and 
\[
\beta_{\mathscr{O}}(\overline{\op{c}}_3)=\beta_{\mathscr{O}}(\overline{\op{c}}_1\overline{\op{c}}_2)=\beta_{\det\gamma_n^\vee}(1)\beta_{\det\gamma_n^\vee}(\overline{\op{c}}_2).
\]
The remaining torsion relations are not completely spelled out for typesetting reasons.
\end{example}

\begin{proposition}
\label{prop:pwhitney}
The homomorphism $\widetilde{\op{CH}}^\bullet({\op{B}}\op{Sp}_{2n})\to\widetilde{\op{CH}}^\bullet({\op{B}}\op{GL}_n,\mathscr{O})$ induced from symplectification maps the odd Pontryagin classes as follows 
\[
\op{p}_{2i+1}\mapsto (\beta_{\mathscr{O}}(\overline{\op{c}}_{2i}))^2 + \op{p}_{2i}\beta_{\mathscr{O}}(\overline{\op{c}}_1)=\beta_{\mathscr{O}}(\overline{\op{c}}_{2i}\overline{\op{c}}_{2i+1}). 
\]

With this notation, the restriction  along the Whitney sum map ${\op{B}}(\op{GL}_n\times \op{GL}_m)\to {\op{B}}\op{GL}_{n+m}$ maps the Pontryagin classes as follows
\[
\op{p}_i\mapsto \sum_{j=\max\{0,i-m\}}^{\min\{i,n\}} \op{p}_j\otimes \op{p}_{i-j}
\]
where the sum is over the indices $j$ such that $\op{p}_j$ and $\op{p}_{i-j}$ are Pontryagin classes for $\op{GL}_n$ and $\op{GL}_m$, respectively.
\end{proposition}

\begin{proof}
By Theorem~\ref{thm:glnin}, it suffices to show that all three classes have the same reduction in $\op{Ch}^\bullet({\op{B}}\op{GL}_n)$. This follows from Proposition~\ref{prop:pwhitneymod2}. The Whitney sum formula then follows directly from the Whitney sum formula for the Pontryagin classes of symplectic bundles and the compatibility of Whitney sum and symplectification, cf. \cite{chow-witt}.
\end{proof}

\begin{remark}
Note that the Whitney sum formula above is exactly the classical one from \cite{brown}. It is easier to state simply by our conventions, cf. \cite[Remark 5.7]{chow-witt}, concerning indexing of the Pontryagin classes. 
\end{remark}

\section{The Chow--Witt ring of \texorpdfstring{${\op{B}}\op{GL}_n$}{BGLn}: proofs}
\label{sec:inproofs}

The main goal of this section is to prove Theorem~\ref{thm:glnin} which is a Chow--Witt analogue of \v Cadek's description of integral cohomology of ${\op{B}}\op{O}(n)$ with local coefficients. The arguments are based on the decomposition into $\mathbf{W}$-cohomology and the image of $\beta$.

\subsection{Projective spaces}

As a first step and basis of the inductive proof we need to recall the computations of the $\mathbf{I}^\bullet$-cohomology and Chow--Witt rings of projective spaces $\mathbb{P}^n$ from \cite{fasel:ij}. Since $\op{Pic}(\mathbb{P}^n)\cong\mathbb{Z}$, there are only two possible dualities to consider, given by the line bundles $\mathscr{O}_{\mathbb{P}^n}$ and $\mathscr{O}_{\mathbb{P}^n}(1)$. 

It is a most classical computation that $\op{Ch}^\bullet(\mathbb{P}^n)\cong\mathbb{Z}/2\mathbb{Z}[\overline{\op{c}}_1]/(\overline{\op{c}}_1^{n+1})$. The Steenrod squares are given by $\op{Sq}^2_{\mathscr{O}}(\overline{\op{c}}_1)=\overline{\op{c}}_1^2$ and $\op{Sq}^2_{\mathscr{O}(1)}(\overline{\op{c}}_1)=0$. In particular, $\ker\op{Sq}^2_{\mathscr{O}}=\mathbb{Z}/2\mathbb{Z}[\overline{\op{c}}_1^2]$, and the kernel of $\op{Sq}^2_{\mathscr{O}(1)}$ is the submodule of $\op{Ch}^\bullet(\mathbb{P}^n)$ generated by odd powers of $\overline{\op{c}}_1$. 

The following is a direct reformulation of the computations in \cite[Section 11]{fasel:ij}. 

\begin{proposition}
\label{prop:faselpn}
\begin{enumerate}
\item If $n$ is odd, then 
\[
\bigoplus_q\op{H}^q(\mathbb{P}^n,\mathbf{I}^q\oplus\mathbf{I}^q(\det\gamma_1^\vee)) \cong \op{W}(F)[\op{e}_1,\op{R}]/(\op{I}(F)\cdot\op{e}_1, \op{e}_1^{n+1}, \op{e}_1\op{R}, \op{R}^2)
\]
Moreover, $\op{e}_1=\beta_{\mathscr{O}(1)}(1)$ and $\op{R}\in \op{H}^n(\mathbb{P}^n,\mathbf{I}^n)$ is the fundamental class of $\mathbb{P}^n$ (which is orientable in this case). The image of $\op{R}$ under the reduction morphism $\rho$ is $\overline{\op{c}}_n\in\op{Ch}^\bullet({\op{B}}\op{GL}_n)$. 
\item If $n$ is even, then 
\[
\bigoplus_q\op{H}^q(\mathbb{P}^n,\mathbf{I}^q\oplus\mathbf{I}^q(\det\gamma_1^\vee)) \cong \op{W}(F)[\op{e}_1,\op{e}_n^\perp]/(\op{I}(F)\cdot\op{e}_1, \op{e}_1^{n+1},\op{e}_1\op{e}_n^\perp,(\op{e}_n^\perp)^2)
\]
Again, $\op{e}_1=\beta_{\mathscr{O}(1)}(1)$, and the class $\op{e}_n^\perp\in \op{H}^n(\mathbb{P}^n,\mathbf{I}^n(\det\gamma_1^\vee))$ is the Euler class of the rank $n$ hyperplane bundle on $\mathbb{P}^n\cong(\mathbb{P}^n)^\vee$.
\end{enumerate}
\end{proposition}

\begin{proof}
Note that \cite{fasel:ij} only establishes the additive structure statements, not quite the full presentation of the ring structure as formulated. Nevertheless, the statements about the ring structure are basically direct consequences as follows: since we already know some characteristic classes of vector bundles, we obtain a ring homomorphism from our claimed presentation to the cohomology ring of $\mathbb{P}^n$. Additively we also know that the Euler class reduces to $\overline{\op{c}}_1$, in particular the nontriviality of the powers of the Euler class is then immediate and this already deals with all the torsion classes. The statement for the nontorsion classes $\op{R}$ resp. $\op{e}_n^\perp$ follows directly, since these cannot have nontrivial intersections with anything else for dimension reasons. 
\end{proof}

\begin{remark}
The classical presentations of the integral cohomology of real projective spaces are recovered exactly for $F=\mathbb{R}$. The algebraic Euler class maps to the topological Euler class under real realization so that also the real realization morphism induces an isomorphism from $\mathbf{I}^\bullet$-cohomology to the integral cohomology of real projective space, cf.~\cite{4real}. 
\end{remark}

The following is the Chow--Witt version of \cite[Lemma 1]{cadek}. This is basically a direct consequence of the above restatement of the computations in \cite[Section 11]{fasel:ij}, noting that ${\op{B}}\op{GL}_1\cong\mathbb{P}^\infty$.

\begin{proposition}
\label{prop:casen1}
The Euler class $\op{e}_1\in \op{H}^1(\mathbb{P}^\infty,\mathbf{I}^1(\det\gamma_1^\vee))$ is nontrivial. Moreover, $\op{e}_1=\beta_{\det\gamma_1^\vee}(1)$. There is an isomorphism
\[
\bigoplus_q\op{H}^q(\mathbb{P}^\infty,\mathbf{I}^q\oplus \mathbf{I}^q(\det\gamma_1^\vee))\cong \op{W}(F)[\op{e}_1]/(\op{I}(F)\cdot\op{e}_1).
\]
The reduction morphism $\op{H}^1(\mathbb{P}^\infty,\mathbf{I}^1(\det\gamma_1^\vee))\to \op{Ch}^1(\mathbb{P}^\infty)$ maps $\op{e}_1$ to $\op{Sq}^2_{\det\gamma_1^\vee}(1)=\overline{\op{c}}_1$.
In particular, Theorem~\ref{thm:glnin} is true for $n=1$. 
\end{proposition}

\begin{remark}
  Alternatively, we can formulate the description of the $\mathbf{I}$-cohomology of projective space in terms of the decomposition into $\mathbf{W}$-cohomology and the image of $\beta$. The $\mathbf{W}$-cohomology of $\mathbb{P}^n$ is an exterior $\op{W}(F)$-algebra on one generator, which is $\op{e}_n^\perp\in\op{H}^n(\mathbb{P}^n,\mathbf{W}(\mathscr{O}(1)))$ for $n$ even and $\op{R}=[\op{pt}]\in\op{H}^n(\mathbb{P}^n,\mathbf{W})$ for $n$ odd. The image of $\beta_{\mathscr{L}}$ is identified with the image of $\op{Sq}^2_{\mathscr{L}}$ and consists of the appropriate powers of $\op{e}_1$. Multiplication with torsion classes can be computed after reduction in $\op{Ch}^\bullet(\mathbb{P}^n)$.
\end{remark}

\subsection{Computation of $\mathbf{W}$-cohomology}

The next step is the computation of the $\mathbf{W}$-cohomology of ${\op{B}}\op{GL}_n$. 

\begin{proposition}
  \label{prop:glnw}
  The $\mathbf{W}$-cohomology of ${\op{B}}\op{GL}_n$ is given as follows:
  \[
  \op{H}^\bullet({\op{B}}\op{GL}_n,\mathbf{W}\oplus\mathbf{W}(\det\gamma_n^\vee))\cong \left\{\begin{array}{ll}
  \op{W}(F)[{\op{p}}_2,{\op{p}}_4,\dots,{\op{p}}_{n-2},\op{e}_n] & n\equiv 0\bmod 2 \\
  \op{W}(F)[{\op{p}}_2,{\op{p}}_4,\dots,{\op{p}}_{n-1}] & n\equiv 1\bmod 2 
  \end{array}\right.
  \]

  The morphisms $\op{H}^\bullet({\op{B}}\op{GL}_n,\mathbf{W}(\mathscr{L}))\to \op{H}^\bullet({\op{B}}\op{GL}_{n-1},\mathbf{W}(\mathscr{L})$ induced by the stabilization morphism ${\op{B}}\op{GL}_{n-1}\to {\op{B}}\op{GL}_n$ are compatible with Pontryagin classes. The restriction along ${\op{B}}\op{GL}_{2n+1}\to {\op{B}}\op{GL}_{2n}$  maps ${\op{p}}_{2n}$ to ${\op{e}}_{2n}^2$.
\end{proposition}

\begin{proof}
  We note that the compatibility of the Pontryagin classes with stabilization follows from their definition, cf.~\cite[Proposition 5.8]{chow-witt}. 
  
  The result is proved by induction. The base case for the induction is given by ${\op{B}}\op{GL}_1\cong\mathbb{P}^\infty$. In this case, the claim is that
  \[
  \op{H}^q(\mathbb{P}^\infty,\mathbf{W}(\mathscr{L}))\cong\left\{\begin{array}{ll}\op{W}(F) & q=0, \mathscr{L}=\mathscr{O}\\0&\textrm{otherwise}\end{array}\right.
  \]
  This follows from Fasel's computations, cf. Proposition~\ref{prop:casen1}.

  For the inductive step, we use the localization sequence of Proposition~\ref{prop:locgln}
\[
\cdots\to \op{H}^{q-n}({\op{B}}\op{GL}_n,\mathbf{W}(\mathscr{L}\otimes \det\gamma_n))\xrightarrow{\op{e}_n} \op{H}^q({\op{B}}\op{GL}_n,\mathbf{W}(\mathscr{L}))\xrightarrow{\iota^\ast}
\]
\[
\xrightarrow{\iota^\ast} \op{H}^q({\op{B}}\op{GL}_{n-1},\mathbf{W}(\iota^\ast\mathscr{L}))\xrightarrow{\partial} \op{H}^{q-n+1}({\op{B}}\op{GL}_n,\mathbf{W}(\mathscr{L}\otimes \det\gamma_n))\to\cdots
\]

If $n$ is even, then by the induction hypothesis $\op{H}^\bullet({\op{B}}\op{GL}_{n-1},\mathbf{W}(\mathscr{L}))$ is a polynomial $\op{W}(F)$-algebra generated by the Pontryagin classes $\op{p}_2,\dots,\op{p}_{n-2}$. Since the stabilization morphism $\iota^\ast$ is compatible with the Pontryagin classes, it is surjective, hence $\partial=0$. Thus, $\op{e}_n$ is injective. Induction on the cohomological degree proves the claim that $\op{e}_n$ is a new polynomial generator; alternatively, we can use the splitting principle of \cite[Proposition 7.8]{chow-witt} to show independence of $\op{e}_n$ from the Pontryagin classes. 

If $n$ is odd, we know that $\op{e}_n=0$ in $\mathbf{W}$-cohomology, since by Proposition~\ref{prop:eulerrel} it is in the image of $\beta$. Therefore, the boundary map
\[
\partial\colon \op{H}^{n-1}({\op{B}}\op{GL}_{n-1},\mathbf{W}(\det\gamma_{n-1}^\vee))\to \op{H}^0({\op{B}}\op{GL}_n,\mathbf{W})
\]
is surjective. The target is a cyclic $\op{W}(F)$-module generated by $1$, and by the inductive assumption the image is a cyclic $\op{W}(F)$-module generated by $\partial{\op{e}}_{n-1}$. In particular, $\partial{\op{e}}_{n-1}=1$, up to a unit in $\op{W}(F)$. By the derivation property for $\partial$, the boundary map is trivial on $\op{H}^\bullet({\op{B}}\op{GL}_{n-1},\mathbf{W})$ and injective on $\op{H}^\bullet({\op{B}}\op{GL}_{n-1},\mathbf{W}(\det\gamma_{n-1}^\vee))$. This implies that the $\mathbf{W}$-cohomology of ${\op{B}}\op{GL}_n$ is a polynomial $\op{W}(F)$-algebra generated by the Pontryagin classes $\op{p}_2,\dots,\op{p}_{n-1}$.

Finally, to prove the claim concerning restriction of the top Pontryagin class, consider the morphism 
\[
o^\ast\colon \op{H}^\bullet({\op{B}}\op{GL}_{2n(+1)}, \mathbf{W}(\det\gamma_n^\vee)) \to \op{H}^\bullet({\op{B}}\op{SL}_{2n(+1)}, \mathbf{W})
\]
given by pullback to  the orientation cover. This maps the Pontryagin classes amnd Euler class to their respective counterparts for ${\op{B}}\op{SL}_{2n(+1)}$. From the present computation of the $\mathbf{W}$-cohomology of ${\op{B}}\op{GL}_{2n(+1)}$ and the computations in \cite[Theorem 1.3]{chow-witt} for ${\op{B}}\op{Sl}_{2n(+1)}$ we conclude that $o^\ast$ is injective. Moreover, $\op{p}_{2n}-\op{e}_{2n}^2$ is mapped to $0$ by \cite[Theorem 1.3]{chow-witt} which proves the claim.
\end{proof}

\begin{remark}
  For the case ${\op{B}}\op{SL}_n$, the analogous formulas can be obtained from the general machinery for $\eta$-inverted cohomology theories in  \cite{ananyevskiy}.
\end{remark}

\subsection{Relations in the mod 2 Chow ring}

In this subsection we show that the ideal $\mathscr{I}_n$ of relations between characteristic classes is annihilated by the composition
\[
\mathscr{R}_n\xrightarrow{\theta_n} \op{H}^\bullet({\op{B}}\op{GL}_n,\mathbf{I}^\bullet\oplus\mathbf{I}^\bullet(\det\gamma_n^\vee)) \xrightarrow{\rho}\op{Ch}^\bullet({\op{B}}\op{GL}_n)^{\oplus 2}.
\]

\begin{lemma}
\label{lem:oddeulermod2}
Assume $n$ is odd. With the above notation we have
\[
\rho(\op{e}_n)=\rho\circ \beta_{\det\gamma_n^\vee}(\overline{\op{c}}_{n-1})=\overline{\op{c}}_{n}.
\]
\end{lemma}

\begin{proof}
This follows from \cite[Proposition 10.3, Remark 10.5]{fasel:ij}, the  identification $\op{Sq}^2_{\det\gamma_n^\vee}=\rho\circ\beta_{\det\gamma_n^\vee}$ from Proposition~\ref{prop:twistedsq2}, and the identification of Stiefel--Whitney classes with reductions of Chern classes in Proposition~\ref{prop:sw}.
\end{proof}

\begin{proposition}
\label{prop:type3mod2}
For two index sets $J$ and $J'$, the elements 
\begin{eqnarray*}
B_J\cdot B_{J'}&-&\sum_{k\in J} B_{\{k\}}\cdot P_{(J\setminus\{k\})\cap J'}\cdot B_{\Delta(J\setminus\{k\},J')}\\
B_J\cdot T_{J'}&-&\sum_{k\in J}B_{\{k\}}\cdot P_{(J\setminus\{k\})\cap J'}\cdot T_{\Delta(J\setminus\{k\},J')}\\
T_J\cdot B_{J'}&-&B_J\cdot T_{J'} + T_\emptyset \cdot P_{J\cap J'}\cdot B_{\Delta(J,J')}\\
T_J\cdot T_{J'}&-&B_J\cdot B_{J'}+T_\emptyset\cdot P_{J\cap J'}\cdot T_{\Delta(J,J')}
\end{eqnarray*}

have trivial images under the composition $\rho\circ\theta_n\colon \mathscr{R}_n\to\op{Ch}^\bullet({\op{B}}\op{GL}_n)$.
\end{proposition}

\begin{proof}
The first relation can be established as in \cite[Proposition 7.13]{chow-witt}. Note that $\rho\circ \theta_n$ maps the elements $B_J$ and $T_J$ to the elements $\op{Sq}^2_{\mathscr{O}}(\overline{\op{c}}_{2j_1}\cdots \overline{\op{c}}_{2j_k})$ and $\op{Sq}^2_{\det\gamma_n^\vee}(\overline{\op{c}}_{2j_1}\cdots \overline{\op{c}}_{2j_k})$, respectively, cf. Proposition~\ref{prop:twistedsq2}.  The proofs of the other relations can be done by the same manipulations as detailed in \cite[Lemma 4]{cadek}. 
\end{proof}

\begin{corollary}
\label{cor:slninmod2}
The composition $\rho\circ\theta_n\colon \mathscr{R}_n\to\op{Ch}^\bullet({\op{B}}\op{GL}_n)^{\oplus 2}$ factors through the quotient $\mathscr{R}_n/\mathscr{I}_n$. 
\end{corollary}

\begin{proof}
This follows directly from Lemma~\ref{lem:oddeulermod2} and Proposition~\ref{prop:type3mod2}. 
\end{proof}

\begin{proposition}
\label{prop:pwhitneymod2}
Let $2i+1\leq n$ be an odd natural number. Then 
\[
\rho(\op{p}_{2i+1})=(\op{Sq}^2_{\mathscr{O}}(\overline{\op{c}}_{2i}))^2+ \rho(\op{p}_{2i})\op{Sq}^2_{\mathscr{O}}(\overline{\op{c}}_1)=\op{Sq}^2_{\mathscr{O}}(\overline{\op{c}}_{2i}\overline{\op{c}}_{2i+1}) 
\]
\end{proposition}

\begin{proof}
The claim follows from the computations below, cf. \cite[p. 288]{brown}:
\[
\rho(\op{p}_{2i+1})(\mathscr{E}_n) = \overline{\op{c}}_{4i+2}(\mathscr{E}_n\oplus\overline{\mathscr{E}_n}) = \overline{\op{c}}_{4i+2}(\mathscr{E}_n^{\oplus 2})= \overline{\op{c}}_{2i+1}(\mathscr{E}_n)^2. 
\]
\[
(\op{Sq}^2_{\mathscr{O}}(\overline{\op{c}}_{2i}))^2+ \rho(\op{p}_{2i})\op{Sq}^2_{\mathscr{O}}(\overline{\op{c}}_1) = (\overline{\op{c}}_{2i+1}+\overline{\op{c}}_1\overline{\op{c}}_{2i})^2+\overline{\op{c}}_{2i}^2\overline{\op{c}}_1^2=\overline{\op{c}}_{2i+1}^2.
\]
\[
\op{Sq}^2_{\mathscr{O}}(\overline{\op{c}}_{2i}\overline{\op{c}}_{2i+1})=\overline{\op{c}}_{2i}\op{Sq}^2_{\mathscr{O}}(\overline{\op{c}}_{2i+1})+\overline{\op{c}}_{2i+1}\op{Sq}^2_{\mathscr{O}}(\overline{\op{c}}_{2i})=\overline{\op{c}}_{2i+1}^2.\qedhere
\]
\end{proof}

\subsection{Proof of Theorem~\ref{thm:glnin}}
We first note that Proposition~\ref{prop:glnw}, in combination with Lemma~\ref{lem:wsplit}, a priori implies a splitting of $\mathbf{I}$-cohomology into $\mathbf{W}$-cohomology and the image of $\beta$, and this is the key tool in the proof. This already establishes part (3) of the theorem. 

Part (2) of the theorem follows from Lemma~\ref{prop:twistedsq2} for the Bockstein classes and \cite[Corollary 7.11]{chow-witt} for the Pontryagin and Euler classes.

To prove part (1) of the theorem, consider the ring homomorphism
\[
\theta_n\colon\mathscr{R}_n\to \bigoplus_{q,\mathscr{L}} \op{H}^q({\op{B}}\op{GL}_n,\mathbf{I}^q(\mathscr{L}))
\]
defined in Theorem~\ref{thm:glnin}. The first step is to show that $\theta_n$ factors through the quotient $\mathscr{R}_n/\mathscr{I}_n$, i.e., that $\theta_n(\mathscr{I}_n)=0$. We consider the relations generating $\mathscr{I}_n$ given in Definition~\ref{def:rngln}. Relations of type (1) hold in by Lemma~\ref{lem:baertorsion}, relations of type (2) by Proposition~\ref{prop:eulerrel}. Relations of type (3) are annihilated by the composition $\rho\circ\theta_n\colon \mathscr{R}_n\to \op{Ch}^\bullet({\op{B}}\op{GL}_n)$ by Proposition~\ref{prop:type3mod2}. By Proposition~\ref{prop:glnw}, the $\mathbf{W}$-cohomology of ${\op{B}}\op{GL}_n$ is free, hence Lemma~\ref{lem:wsplit} implies that the reduction $\rho\colon \op{H}^q({\op{B}}\op{GL}_n,\mathbf{I}^q(\mathscr{L}))\to \op{Ch}^q({\op{B}}\op{GL}_n)$ is injective on the image of $\beta_{\mathscr{L}}$. Since all relations of type (3) are in the image of $\beta_{\mathscr{L}}$, those relations have trivial image under $\theta_n$. Therefore, we get a well-defined ring homomorphism
\[
\overline{\theta}_n\colon\mathscr{R}_n/\mathscr{I}_n\to \bigoplus_{q,\mathscr{L}} \op{H}^q({\op{B}}\op{GL}_n,\mathbf{I}^q(\mathscr{L}))
\]

We now prove that the ring homomorphism $\overline{\theta}_n$ is surjective. First, we note that $\overline{\theta}_n$ surjects onto $\op{Im}\beta_{\mathscr{L}}$ if and only if the composition
\[
\rho\circ\overline{\theta}_n\colon \mathscr{R}_n/\mathscr{I}_n\to \bigoplus_{q,\mathscr{L}} \op{H}^q({\op{B}}\op{GL}_n,\mathbf{I}^q(\mathscr{L}))\to \op{Ch}^\bullet({\op{B}}\op{GL}_n) 
\]
surjects onto the image of $\op{Sq}^2_{\mathscr{L}}\colon\op{Ch}^{\bullet-1}({\op{B}}\op{GL}_n)\to \op{Ch}^\bullet({\op{B}}\op{GL}_n)$. By Corollary~\ref{cor:wuimage}, we know that the image of $\op{Sq}^2_{\mathscr{L}}$ is contained in the subring generated by the classes $\op{Sq}^2_{\mathscr{L}}(\overline{\op{c}}_{2j_1}\cdots\overline{\op{c}}_{2j_l})$, $\op{Sq}^2_{\det\gamma_n}(1)$, $\overline{\op{c}}_{2i}^2$ and $\overline{\op{c}}_n$. By part (2) of the theorem, all these classes are reductions of classes in the image of $\theta_n$, proving that $\overline{\theta}_n$ surjects onto the image of $\beta$. It then suffices to show that the composition
\[
 \mathscr{R}_n/\mathscr{I}_n\xrightarrow{\overline{\theta}_n} \bigoplus_{q,\mathscr{L}} \op{H}^q({\op{B}}\op{GL}_n,\mathbf{I}^q(\mathscr{L}))\to \bigoplus_{q,\mathscr{L}}\op{H}^q({\op{B}}\op{GL}_n,\mathbf{W}(\mathscr{L}))
\]
is surjective, where the second map is the projection onto $\mathbf{W}$-cohomology. But this follows from Proposition~\ref{prop:glnw}, finishing the surjectivity proof. 

Finally, we prove that $\overline{\theta}_n$ is injective. First, we consider the $\op{W}(F)$-torsion-free part of $\mathscr{R}_n/\mathscr{I}_n$ which is generated, as commutative graded $\op{W}(F)$-algebra, by the $P_i$, and $X_{2n}$ if applicable. The restriction of $\overline{\theta}_n$ to that subalgebra is injective by Proposition~\ref{prop:glnw}. The injectivity on the torsion part, i.e., the ideal generated by the classes $B_J$, $T_J$ for $J=\{j_1,\dots,j_l\}$ and $T_\emptyset$ can be checked after composition with $\rho$, by the decomposition of Lemma~\ref{lem:wsplit} (and Proposition~\ref{prop:glnw}) and the resulting fact that $\rho$ is injective on the image of $\beta$. The direct translation (replacing $\op{w}_i$ by $\overline{\op{c}}_i$ and $\op{Sq}^1$ by $\op{Sq}^2$) of the argument on p. 285 of \cite{cadek} takes care of that, cf. also \cite[Proposition 8.15]{chow-witt}.

\section{Chow--Witt rings of finite Grassmannians: statement of results}
\label{sec:chowwittgrass}

In the following two sections, we compute the Chow--Witt rings of the finite Grassmannians $\op{Gr}(k,n)$. The results are stated in the present section, and the proofs are deferred to the next section. 

\subsection{Generators from characteristic classes}

The first step is to get enough classes in $\widetilde{\op{CH}}^\bullet(\op{Gr}(k,n),\mathscr{L})$. We realize the Grassmannian $\op{Gr}(k,n)$ over the field $F$ as the variety of $k$-dimensional $F$-subspaces of $V=F^n$. Recall that we have an exact sequence of vector bundles on $\op{Gr}(k,n)$: 
\[
0\to \mathscr{S}_k\to \mathscr{O}_{\op{Gr}(k,n)}^{\oplus n}\to \mathscr{Q}_{n-k}\to 0,
\]
Here, $\mathscr{S}_k$ is the \emph{tautological subbundle}, mapping a point $[W]$ corresponding to a $k$-dimensional subspace $W\subset V$ to $W$, and $\mathscr{Q}_{n-k}$ is the \emph{tautological quotient bundle}, mapping a point $[W]$ to the quotient space $V/W$.

There is a vector bundle torsor  $f\colon \op{GL}_n/\op{GL}_k\times\op{GL}_{n-k}\to \op{Gr}(k,n)$ over the Grassmannian. This is an $\mathbb{A}^1$-weak equivalence, and the above exact sequence of vector bundles splits over $\op{GL}_n/\op{GL}_k\times\op{GL}_{n-k}$. Consequently, we obtain an $\mathbb{A}^1$-fiber sequence 
\[
\op{GL}_n/\op{GL}_k\times\op{GL}_{n-k}\to {\op{B}}\op{GL}_k\times{\op{B}}\op{GL}_{n-k}\xrightarrow{\oplus} {\op{B}}\op{GL}_n
\]
where the second map is the Whitney sum map and the first map classifies the pair $(f^\ast\mathscr{S}_k,f^\ast\mathscr{Q}_{n-k})$. We can also consider the map $c\colon\op{Gr}(k,n)\to {\op{B}}\op{GL}_k\times{\op{B}}\op{GL}_{n-k}$ obtained by composing a homotopy inverse of $f$ with the inclusion of the homotopy fiber, and this map classifies the pair $(\mathscr{S}_k,\mathscr{Q}_{n-k}$).

Note that there are two possible dualities on ${\op{B}}\op{GL}_k$, corresponding to the line bundles $\mathscr{O}$ and $\det\gamma_k^\vee$; and similarly there are two possible dualities on ${\op{B}}\op{GL}_{n-k}$ corresponding to $\mathscr{O}$ and $\det\gamma_{n-k}^\vee$. Consequently, there are four possible dualities on ${\op{B}}\op{GL}_k\times{\op{B}}\op{GL}_{n-k}$, given by the four possible exterior products of the above line bundles. For any choice of line bundles $\mathscr{L}_k$ and $\mathscr{L}_{n-k}$ on ${\op{B}}\op{GL}_k$ and ${\op{B}}\op{GL}_{n-k}$, respectively, the classifying map $c$ above induces homomorphisms of Chow--Witt groups
\[
\widetilde{\op{CH}}^\bullet({\op{B}}\op{GL}_k\times{\op{B}}\op{GL}_{n-k}, \mathscr{L}_k\boxtimes \mathscr{L}_{n-k}) \to \widetilde{\op{CH}}^\bullet(\op{Gr}(k,n),c^\ast(\mathscr{L}_k\boxtimes \mathscr{L}_{n-k})). 
\]
Note that the bundle $c^\ast(\mathscr{L}_k\boxtimes \mathscr{L}_{n-k})$ is trivial if and only if $\mathscr{L}_k$ and $\mathscr{L}_{n-k}$ are either both trivial or both non-trivial. This follows from the fact that the assignment $(\mathscr{L}_k,\mathscr{L}_{n-k})\mapsto c^\ast(\mathscr{L}_k\boxtimes \mathscr{L}_{n-k})$ can be computed by pulling back both line bundles to the Grassmannian and then taking the tensor product, hence it induces the addition 
\[
\mathbb{Z}/2\mathbb{Z}^{\oplus 2}\cong \op{Ch}^1({\op{B}}\op{GL}_k\times{\op{B}}\op{GL}_{n-k})\to \op{Ch}^1(\op{Gr}(k,n)) \cong \mathbb{Z}/2\mathbb{Z}.
\]  
The induced homomorphisms assemble into a ring homomorphism of the total Chow--Witt rings (to the extent that this makes sense, cf. the remarks on~\cite{lax:similitude} in Section~\ref{sec:recall}).

In particular, the images of the characteristic classes for vector bundles, cf. Section~\ref{sec:vbclasses} and in particular Theorem~\ref{thm:glnin} resp. the main result Theorem~\ref{thm:main1}, induce classes in the Chow--Witt ring of $\op{Gr}(k,n)$. The characteristic classes for the tautological subbundle $\mathscr{S}_k$ are 
\begin{enumerate}
\item the Pontryagin classes $\op{p}_1,\op{p}_2,\dots,\op{p}_{k-1}$,
\item the Euler class $\op{e}_k$,
\item the (twisted) Bockstein classes $\beta_{\mathscr{O}}(\overline{\op{c}}_{2j_1}\cdots\overline{\op{c}}_{2j_l})$ and $\beta_{\det\gamma_k^\vee}(\overline{\op{c}}_{2j_1}\cdots\overline{\op{c}}_{2j_l})$, and 
\item the Chern classes $\op{c}_i$.
\end{enumerate}
Similarly, for the tautological quotient bundle $\mathscr{Q}_{n-k}$, we have the same characteristic classes (with different index sets); these will be denoted by an additional superscript $(-)^\perp$.\footnote{The notation is suggestive that $\mathscr{Q}_{n-k}$ is the complement of $\mathscr{S}_k$ in $\mathscr{O}^{\oplus n}$, after pulling back to $\op{GL}_n/\op{GL}_k\times\op{GL}_{n-k}$.} This provides a number of canonical elements in $\widetilde{\op{CH}}^\bullet(\op{Gr}(k,n),\mathscr{L})$. It turns out that in the cases where $\dim\op{Gr}(k,n)=k(n-k)$ is even, these classes generate the Chow--Witt ring; in the case where the dimension is odd, there is essentially one additional class arising as lift of an Euler class.

\begin{remark}
Note that we follow the convention of \cite[Remark 5.7]{chow-witt}, including all Pontryagin classes without added signs or reindexing. While the odd Pontryagin classes are $\op{I}(F)$-torsion and can be expressed in terms of Bockstein classes, this convention makes the Whitney sum formula for Pontryagin classes easier to state, cf. Proposition~\ref{prop:pwhitney} and the subsequent remark.
\end{remark}

\subsection{Chow rings of Grassmannians}

Before giving the statement concerning the structure of the Chow--Witt rings, we discuss the Chow rings of the Grassmannians. This result is very well-known and can be found in the relevant books on intersection theory, such as \cite{3264}. 

\begin{proposition}
\label{prop:chowgrass}
Let $F$ be an arbitrary field. Let $\op{c}_1,\dots,\op{c}_k$ be the Chern classes of the tautological rank $k$ subbundle, and let $\op{c}^\perp_1,\dots,\op{c}^\perp_{n-k}$ be the Chern classes of the tautological rank $(n-k)$ quotient bundle. Then there is a canonical isomorphism
\[
\op{CH}^\bullet(\op{Gr}(k,n))\cong\mathbb{Z}[\op{c}_1,\dots,\op{c}_k,\op{c}^\perp_1,\dots,\op{c}^\perp_{n-k}]/(\op{c}\cdot\op{c}^\perp=1),
\]
where $\op{c}=\sum_{i=0}^k\op{c}_i$ and $\op{c}^\perp=\sum_{i=0}^{n-k}\op{c}_i^\perp$ are the total Chern classes.
\end{proposition}

Note that the relation $\op{c}\cdot\op{c}^\perp=1$ is very natural: $\op{c}$ is the total Chern class of the tautological bundle $\mathscr{S}_k$, $\op{c}^\perp$ is the total Chern class of  $\mathscr{Q}_{n-k}$, and the relation $\op{c}\cdot\op{c}^\perp=1$ is just the Whitney sum formula for the extension
\[
0\to \mathscr{S}_k\to \mathscr{O}_{\op{Gr}(k,n)}^{\oplus n}\to \mathscr{Q}_{n-k}\to 0.
\]


\begin{proposition}
\label{cor:chowmod2grass}
Let $F$ be an arbitrary field. Let $\overline{\op{c}}_1,\dots,\overline{\op{c}}_k$ be the Stiefel--Whitney classes of the tautological rank $k$ bundle, and let $\overline{\op{c}}^\perp_1,\dots,\overline{\op{c}}^\perp_{n-k}$ be the Stiefel--Whitney classes of the tautological rank $(n-k)$ quotient bundle. 
\begin{enumerate}
\item  There is a canonical isomorphism
\[
\op{Ch}^\bullet(\op{Gr}(k,n))\cong\mathbb{Z}/2\mathbb{Z}[\overline{\op{c}}_1,\dots,\overline{\op{c}}_k,\overline{\op{c}}^\perp_1,\dots,\overline{\op{c}}^\perp_{n-k}]/(\overline{\op{c}}\cdot\overline{\op{c}}^\perp=1).
\]
where $\overline{\op{c}}=\sum_{i=0}^k\overline{\op{c}}_i$ and $\overline{\op{c}}^\perp=\sum_{i=0}^{n-k}\overline{\op{c}}_i^\perp$ are the total Stiefel--Whitney classes.
\item 
The Steenrod square is given by 
\[
\op{Sq}^2_{\mathscr{O}}\colon \op{Ch}^\bullet(\op{Gr}(k,n))\to \op{Ch}^\bullet(\op{Gr}(k,n))\colon  \overline{\op{c}}^{(\perp)}_{j}\mapsto \overline{\op{c}}_1^{(\perp)} \overline{\op{c}}^{(\perp)}_{j} + (j-1) \overline{\op{c}}_{j+1}^{(\perp)}
\]
\item 
The twisted Steenrod square is given by 
\[
\op{Sq}^2_{\det\mathscr{S}_k^\vee}\colon \op{Ch}^\bullet(\op{Gr}(k,n))\to \op{Ch}^\bullet(\op{Gr}(k,n))\colon  \overline{\op{c}}^{(\perp)}_{j}\mapsto (j-1)\overline{\op{c}}_{j+1}^{(\perp)}.
\]
\end{enumerate}
\end{proposition}


The following consequence of the description of the Chow ring given in Proposition~\ref{cor:chowmod2grass} will be relevant later. 

\begin{proposition}
\label{prop:kercokerchowmod2}
Let $1\leq k<n$ and consider the ring  
\[
A=\mathbb{Z}/2\mathbb{Z}[\overline{\op{c}}_1,\dots,\overline{\op{c}}_k,\overline{\op{c}}^\perp_1,\dots,\overline{\op{c}}^\perp_{n-k}]/(\overline{\op{c}}\cdot\overline{\op{c}}^\perp=1). 
\]
\begin{enumerate}
\item The kernel of the multiplication by $\overline{\op{c}}_{n-k}^\perp$ is the ideal $\langle \overline{\op{c}}_k\rangle\subseteq A$.
\item The cokernel of the multiplication by $\overline{\op{c}}_{n-k}^\perp$ is 
\[
A/\langle \overline{\op{c}}_{n-k}^\perp\rangle\cong \mathbb{Z}/2\mathbb{Z}[\overline{\op{c}}_1,\dots,\overline{\op{c}}_k,\overline{\op{c}}^\perp_1,\dots,\overline{\op{c}}^\perp_{n-k-1}]/(\overline{\op{c}}\cdot\overline{\op{c}}^\perp=1).
\]
\end{enumerate}
\end{proposition}

\begin{proof}
(2) The statement about the cokernel being $A/\langle\overline{\op{c}}_{n-k}^\perp\rangle$ is clear. The explicit description of the algebra also follows directly; note that the relation $\overline{\op{c}}\cdot\overline{\op{c}}^\perp=1$ refers to the total class $\overline{\op{c}}^\perp=1+\overline{\op{c}}_1^\perp+\cdots+\overline{\op{c}}_{n-k-1}^\perp$. 

(1) Clearly, $\langle\overline{\op{c}}_k\rangle\subseteq \ker \overline{\op{c}}_{n-k}^\perp$ since $\overline{\op{c}}_k\overline{\op{c}}_{n-k}^\perp=0$ follows from the Whitney sum relation. The reverse inclusion can be seen e.g. by a dimension count in the kernel-cokernel exact sequence for multiplication by $\overline{\op{c}}_{n-k}^\perp$. 
\end{proof}

\begin{remark}
This is also the formula for the mod 2 cohomology of the real Grassmannians, cf. e.g. \cite{milnor:stasheff}. The notation for the classes $\overline{\op{c}}_i$ and $\overline{\op{c}}^\perp_i$ is due to the fact that these are the mod 2 reductions of the Chern classes from the Chow ring. In the real realization these would go exactly to the corresponding Stiefel--Whitney classes. \end{remark}

\subsection{Statement of the main results}

Now we are ready to state the main results describing the $\mathbf{I}^\bullet$-cohomology of the finite Grassmannians. The lengthy formulation boils down to ``the characteristic classes of the tautological bundle and its complement generate the cohomology (except for a new class in when $k(n-k)$ is odd) and the only new relations come from the Whitney sum formula''. 

\begin{theorem}
\label{thm:ingrass}
Let $F$ be a perfect field of characteristic $\neq 2$, and let $1\leq k< n$. The cohomology ring $\bigoplus_q\op{H}^q(\op{Gr}(k,n),\mathbf{I}^q\oplus\mathbf{I}^q(\det\mathscr{S}_k^\vee))$ is isomorphic to the $\mathbb{Z}\oplus\mathbb{Z}/2\mathbb{Z}$-graded $\op{W}(F)$-algebra generated by 
\begin{enumerate}
\item the Pontryagin classes $\op{p}_1,\op{p}_2,\dots,\op{p}_{k}$ of the tautological rank $k$ subbundle and the Pontryagin classes $\op{p}^\perp_1, \op{p}_2^\perp, \dots, \op{p}_{n-k}^\perp$ of the tautological rank $(n-k)$ quotient bundle, where the class $\op{p}_{i}^{(\perp)}$ sits in degree $(2i,0)$,
\item the Euler classes $\op{e}_k$ and $\op{e}_{n-k}^\perp$, sitting in degrees $(k,1)$ and $(n-k,1)$, respectively,
\item for every set $J=\{j_1,\dots,j_l\}$ of natural numbers $0<j_1<\cdots<j_l\leq[(k-1)/2]$, possibly empty, there are Bockstein classes $\beta_J=\beta_{\mathscr{O}}(\overline{c}_{2j_1}\cdots\overline{c}_{2j_l})$ and $\tau_J=\beta_{\det\mathscr{S}_k}(\overline{c}_{2j_1}\cdots\overline{c}_{2j_l})$ in degrees $(d,0)$ and $(d,1)$, respectively, where $d=1+2\sum_{i=1}^lj_i$, 
\item for every set $J=\{j_1,\dots,j_l\}$ of natural numbers $0<j_1<\cdots<j_l\leq[(n-k-1)/2]$, possibly empty, there are Bockstein classes $\beta^\perp_J=\beta_{\mathscr{O}}(\overline{c}^\perp_{2j_1}\cdots\overline{c}^\perp_{2j_l})$ and $\tau^\perp_J=\beta_{\det\mathscr{S}_k}(\overline{c}^\perp_{2j_1}\cdots\overline{c}^\perp_{2j_l})$ in degrees $(d,0)$ and $(d,1)$, respectively, where $d=1+2\sum_{i=1}^lj_i$, 
\item if $k(n-k)$ is odd, we have a class $\op{R}$ in degree $(n-1,0)$
\end{enumerate}
subject to the following relations:
\begin{enumerate}
\item the classes $\op{p}_i,\op{e}_k$, $\beta_J$ and $\tau_J$ satisfy the relations holding in $\op{H}^\bullet({\op{B}}\op{GL}_k)$; the classes $\op{p}_i^\perp,\op{e}_{n-k}^\perp$, $\beta^\perp_J$ and $\tau^\perp_J$ satisfy the relations in $\op{H}^\bullet({\op{B}}\op{GL}_{n-k})$, cf. Theorem~\ref{thm:glnin}.
\item $\op{p}\cdot\op{p}^\perp=1$, i.e., the product of the total Pontryagin classes is $1$.
\item $\op{e}_k\cdot\op{e}_{n-k}^\perp=0$. 
\item $\beta(\overline{\op{c}}\cdot\overline{\op{c}}^\perp)=1$ and $\tau(\overline{\op{c}}\cdot\overline{\op{c}}^\perp)=\tau_\emptyset=\tau_\emptyset^\perp$, i.e., applying the (twisted) Bockstein to the product of the total Stiefel--Whitney classes in $\op{Ch}^\bullet$ is trivial.
\item $\op{R}^2=0$, and the product of $\op{R}$ with an $\op{I}(F)$-torsion class $\alpha$ is zero if and only if $\overline{\op{c}}_{k-1}\overline{\op{c}}_{n-k}^\perp\rho(\alpha)=0$.
\end{enumerate}
\end{theorem}

\begin{remark}
Note that we include all Pontryagin classes because this makes the relations from the Whitney sum formula easier to state. 
\end{remark}


\begin{proposition}
\label{prop:inredgrass}
Let $F$ be a perfect field of characteristic $\neq 2$, and let $1\leq k< n$. The reduction homomorphism 
\[
\rho\colon \bigoplus_q\op{H}^q(\op{Gr}(k,n),\mathbf{I}^q\oplus\mathbf{I}^q(\det\mathscr{S}_k)) \to \op{Ch}^q(\op{Gr}(k,n))
\]
is given by 
\begin{eqnarray*}
\op{p}_{2i}^{(\perp)}&\mapsto& (\overline{\op{c}}^{(\perp)}_{2i})^2, \\
\beta_{\mathscr{L}}(\overline{\op{c}}^{(\perp)}_{2j_1}\cdots \overline{\op{c}}^{(\perp)}_{2j_l})&\mapsto&  \op{Sq}^2_{\mathscr{L}}(\overline{\op{c}}^{(\perp)}_{2j_1}\cdots \overline{\op{c}}^{(\perp)}_{2j_l}), \\ \op{e}_k&\mapsto& \overline{\op{c}}_k,\\ \op{e}_{n-k}^\perp&\mapsto& \overline{\op{c}}_{n-k}^\perp,\\
\op{R}&\mapsto & \overline{\op{c}}_{k-1}\overline{\op{c}}_{n-k}^\perp=\overline{\op{c}}_k\overline{\op{c}}_{n-k-1}^\perp.
\end{eqnarray*}
The reduction homomorphism $\rho_{\mathscr{L}}$ is injective on the image of the Bockstein map $\beta_{\mathscr{L}}$. 
\end{proposition}

\begin{remark}
Note that this presentation gives a complete description of the cup product. To multiply two torsion classes, we first rewrite the complementary classes $\overline{\op{c}}^\perp_{2i}$ in terms of polynomials in the ordinary classes $\overline{\op{c}}_{2j}$. (It follows directly from the well-known presentation of $\op{Ch}^\bullet(\op{Gr}(k,n))$ that it is generated by the classes $\overline{\op{c}}_i$ and the complementary classes $\overline{\op{c}}^\perp_j$ can be expressed in terms of these.) The product of classes of the form $\beta_{\mathscr{L}}(\overline{\op{c}}_{2j_1}\cdots\overline{\op{c}}_{2j_l})$ is then given by the relation in $\op{H}^\bullet({\op{B}}\op{GL}_k)$. Note also that the product of $\op{R}$ with an even Pontryagin class is independent from the Pontryagin classes. The product of $\op{R}$ with a torsion class is a torsion class, and so it can be determined by computation in $\op{Ch}^\bullet(\op{Gr}(k,n))$. More detailed descriptions of how to work out products can be found in \cite{schubert}. 
\end{remark}

\begin{theorem}
\label{thm:chowwittgrass}
Let $F$ be a perfect field of characteristic $\neq 2$, and let $1\leq k<n$. Then there is a cartesian square of $\mathbb{Z}\oplus\mathbb{Z}/2\mathbb{Z}$-graded $\op{GW}(F)$-algebras:
\[
\xymatrix{
\widetilde{\op{CH}}^\bullet(\op{Gr}(k,n),\mathscr{O}\oplus\det\mathscr{S}_k^\vee) \ar[r] \ar[d] & \ker \partial_{\mathscr{O}}\oplus \ker\partial_{\det\mathscr{S}_k^\vee} \ar[d]^{\bmod 2}\\
\op{H}^\bullet(\op{Gr}(k,n),\mathbf{I}^\bullet\oplus\mathbf{I}^\bullet(\det\mathscr{S}_k^\vee)) \ar[r]_>>>>>>\rho & \op{Ch}^\bullet(\op{Gr}(k,n))^{\oplus 2}
}
\]
Here $\det\mathscr{S}_k^\vee$ is the determinant of the dual of the tautological rank $k$ subbundle on $\op{Gr}(k,n)$, 
\[
\partial_{\mathscr{L}}\colon \op{CH}^\bullet(\op{Gr}(k,n))\to \op{Ch}^\bullet(\op{Gr}(k,n)) \xrightarrow{\beta_{\mathscr{L}}} \op{H}^{\bullet+1}(\op{Gr}(k,n),\mathbf{I}^{\bullet+1}(\mathscr{L})) 
\]
is the (twisted) integral Bockstein map. 

The kernel of the integral Bockstein map $\partial_{\mathscr{L}}$ is the preimage under reduction mod 2 of the subalgebra of $\op{Ch}^\bullet(\op{Gr}(k,n))$ generated by $({\op{c}_i}^{(\perp)})^2$, ${\op{c}}_k$, ${\op{c}}_{n-k}^\perp$, and ${\op{c}}_k{\op{c}}_{n-k-1}^\perp$ together with the image of $\op{Sq}^2_{\mathscr{L}}$.
\end{theorem}

\begin{proof}
This follows from \cite[Proposition 2.11]{chow-witt} since the Chow ring of $\op{Gr}(k,n)$ is $2$-torsion-free, cf. Proposition~\ref{prop:chowgrass}. The description of $\mathbf{I}^\bullet$-cohomology is given in Theorem~\ref{thm:ingrass}, and the description of the reduction morphism $\rho$ is given in Proposition~\ref{prop:inredgrass}. The description of the kernel of the boundary map follows directly from the definition and the B\"ar sequence, i.e., that the kernel of $\beta_{\mathscr{L}}$ is exactly the image of the reduction map $\rho_{\mathscr{L}}$. 
\end{proof}

\begin{remark}
We can determine the images of Euler classes and Pontryagin classes in Chow theory using Theorem~\ref{thm:glnchw}. 
\end{remark}

\subsection{Examples}

The following are two examples describing the $\mathbf{I}^\bullet$-cohomology of small Grassmannians. For alternative descriptions of the $\mathbf{I}^\bullet$-cohomology, using even Young diagrams for the $\mathbf{W}$-part and checkerboard fillings for Young diagrams for the image of $\beta$, cf.~\cite{schubert}. 

\begin{example}
Let us work out the example case $\op{Gr}(2,4)$. The relevant characteristic classes for $\mathbf{I}^\bullet$-cohomology are the Pontryagin classes 
\[
\op{p}_1^{(\perp)}=\beta_{\mathscr{O}}(\overline{\op{c}}_1^{(\perp)})\in \op{H}^2(\op{Gr}(2,4),\mathbf{I}^2), \quad \op{p}_2^{(\perp)}\in\op{H}^4(\op{Gr}(2,4),\mathbf{I}^4)
\]
and the Euler classes $\op{e}_2,\op{e}_2^\perp\in\op{H}^2(\op{Gr}(2,4), \mathbf{I}^2(\det\mathscr{E}_2^\vee))$; finally, there are all sorts of (twisted) Bockstein classes. 

The relations from the Whitney formula for Stiefel--Whitney classes are 
\[
\overline{\op{c}}_1=\overline{\op{c}}_1^\perp, \quad \overline{\op{c}}_2+\overline{\op{c}}_1^2+\overline{\op{c}}_2^\perp=0, \quad \overline{\op{c}}_1\overline{\op{c}}_2^\perp+\overline{\op{c}}_2\overline{\op{c}}_1^\perp=\overline{\op{c}}_1^3=0, \quad \overline{\op{c}}_2^2+\overline{\op{c}}_2\overline{\op{c}}_1^2=0.
\]
In particular, $\beta(\overline{\op{c}}_i)=\beta(\overline{\op{c}}_i^\perp)$ for $i=1,2$. By the Wu formula, $\op{Sq}^2(\overline{\op{c}}_2)=\overline{\op{c}}_1\overline{\op{c}}_2$ and therefore $\beta(\overline{\op{c}}_1\overline{\op{c}}_2)=0$. Since Bocksteins of squares are trivial by the derivation property, this means that the only nontrivial Bockstein classes are  $\beta_{\mathscr{O}}(\overline{\op{c}}_1)=\op{p}_1$  and $\beta_{\mathscr{O}}(\overline{\op{c}}_2)$. 

With twisted coefficients, we have the class $\beta_{\det\mathscr{E}_2^\vee}(1)$. The other twisted Bockstein classes $\beta_{\det\mathscr{E}_2^\vee}(\overline{\op{c}}_i)$ are trivial:  we can check on reduction mod 2, the case $i=1$ follows directly from the Wu formula and the case $i=2$ follows from $\overline{\op{c}}_3=0$. The other potential $2$-torsion classes are products of the form $\tau_\emptyset\beta$. We check on mod 2 reduction that  $\rho(\tau_\emptyset\beta(\overline{\op{c}}_1))=\overline{\op{c}}_1^3=0$ but $\rho(\tau_\emptyset\beta(\overline{\op{c}}_2))=\overline{\op{c}}_1^2\overline{\op{c}}_2=\overline{\op{c}}_2^2\neq 0$. So the torsion classes with twisted coefficients are $\beta_{\det\mathscr{E}_2^\vee}(1)$ in degree 1 and  $\beta_{\det\mathscr{E}_2^\vee}(1)\beta_{\mathscr{O}}(\overline{\op{c}}_2)$ in degree $4$. 

The relations encoded in $\op{p}\cdot\op{p}^\perp=1$ are the following:
\[
\beta(\overline{\op{c}}_1)=\beta(\overline{\op{c}}_1^\perp), \quad \op{p}_2+\beta(\overline{\op{c}}_1)^2+\op{p}_2^\perp=0, \quad \op{p}_2^2=0.
\]
There is also a relation $2\op{p}_2\beta(\overline{\op{c}}_1)=0$ which is trivially satisfied. From $\overline{\op{c}}_1^3=0$ above we find that $\op{Sq}^2(\overline{\op{c}}_1)^2=\overline{\op{c}}_1^4=0$ and therefore $\beta(\overline{\op{c}}_1)^2=0$. In particular, $\op{p}_2=-\op{p}_2^\perp$. Consequently, the only non-torsion classes are $\op{p}_2$, $\op{e}_2$ and $\op{e}_2^\perp$, with $\op{e}_2^2=(\op{e}_2^\perp)^2=\op{p}_2$. 

Note that $\op{Gr}(2,4)$ is an orientable variety, and Poincar{\'e} duality is satisfied for the $\mathbf{I}^\bullet$-cohomology. 

For the Chow--Witt ring, the Chern classes $\op{c}_1^{(\perp)}$ and  $\op{c}_2^{(\perp)}$ have nontrivial Bocksteins and hence do not lift to the Chow--Witt ring. As noted above, the classes $\op{c}_1^2$, ${\op{c}_1}{\op{c}_2}$ and $\op{c}_2^2$ have trivial Steenrod squares and therefore we have 
\[
\ker\partial_{\mathscr{O}}=\mathbb{Z}[{\op{c}_1}{\op{c}_2},2{\op{c}_i},\op{c}_1^2,\op{c}_2^2].
\]
On the other hand, $\ker\partial_{\det\mathscr{E}_2}$ is the submodule generated by $2$, ${\op{c}_1}$ and $\op{c}_2$.
\end{example}

\begin{example}
\label{ex:gr25}
As another example, we work out the Steenrod squares for $\op{Gr}(2,5)$. The relevant relations arising from the Whitney sum formula are
\[
\overline{\op{c}}_1=\overline{\op{c}}_1^\perp,\qquad \overline{\op{c}}_2^\perp =\overline{\op{c}}_2+\overline{\op{c}}_1^2, \qquad
\overline{\op{c}}_3^\perp = \overline{\op{c}}_1^3, \qquad
\overline{\op{c}}_2^2+\overline{\op{c}}_1^2\overline{\op{c}}_2+\overline{\op{c}}_1^4 = \overline{\op{c}}_2\overline{\op{c}}_1^3=0. 
\]
Now we go through the individual degrees in $\op{Ch}^\bullet(\op{Gr}(2,5))$:
\begin{enumerate}
\item Degree 1 has $\overline{\op{c}}_1$ and $\op{Sq}^2(\overline{\op{c}}_1)=\overline{\op{c}}_1^2$; this class doesn't lift to $\op{H}^1(\op{Gr}(2,5),\mathbf{I}^1)$. 
\item Degree 2 has $\overline{\op{c}}_2$ with $\op{Sq}^2(\overline{\op{c}}_2)=\overline{\op{c}}_1\overline{\op{c}}_2$ and $\overline{\op{c}}_1^2$ with trivial $\op{Sq}^2$. So the latter class lifts to a torsion class $\beta(\overline{\op{c}}_1)\in\op{H}^2(\op{Gr}(2,5),\mathbf{I}^2)$. 
\item Degree 3 has $\overline{\op{c}}_1^3$ with $\op{Sq}^2(\overline{\op{c}}_1^3)=\overline{\op{c}}_1^4$ and $\overline{\op{c}}_1\overline{\op{c}}_2$ with trivial $\op{Sq}^2$. So the latter class lifts to a torsion class $\beta(\overline{\op{c}}_2)\in\op{H}^3(\op{Gr}(2,5),\mathbf{I}^3)$. 
\item Degree 4 has $\overline{\op{c}}_1^4$ and $\overline{\op{c}}_1^2\overline{\op{c}}_2$, both with trivial $\op{Sq}^2$. The class $\overline{\op{c}}_2^2=\overline{\op{c}}_1^4+\overline{\op{c}}_1^2\overline{\op{c}}_2$ lifts to the Pontryagin class, and the $\overline{\op{c}}_1^4$ lifts to $\beta(\overline{\op{c}}_1^3)\in\op{H}^4(\op{Gr}(2,5),\mathbf{I}^4)$. 
\item Degree 5 has $\overline{\op{c}}_1^5$ with $\op{Sq}^2(\overline{\op{c}}_1^5)=\overline{\op{c}}_1^6$ and consequently this class doesn't lift to $\mathbf{I}^5$-cohomology. 
\item Degree 6 has $\overline{\op{c}}_1^6$ with trivial $\op{Sq}^2$, and this class lifts to the integral class $\beta(\overline{\op{c}}_1^5)\in\op{H}^6(\op{Gr}(2,5),\mathbf{I}^6)$
\end{enumerate}
This recovers exactly the pattern for integral cohomology as discussed e.g. in \cite{casian:kodama}. In addition to that, we can use the formulas from Theorem~\ref{thm:main1} to determine the cup product of the torsion classes. Computations as above could be done to determine the cohomology with twisted coefficients as well. 
\end{example}

\section{Chow--Witt rings of finite Grassmannians: proofs}
\label{sec:grassproofs}

In this section, we will now prove the claims about the structure of $\mathbf{I}^\bullet$-cohomology of the Grassmannians discussed in Section~\ref{sec:inproofs}. The overall argument is again to use the decomposition of $\mathbf{I}$-cohomology into the image of $\beta$ and the $\mathbf{W}$-cohomology. We compute the $\mathbf{W}$-cohomology using a version of the inductive procedure used by Sadykov to compute rational cohomology of the real Grassmannians, cf. \cite{sadykov}. The base case $k=1$ is the case of projective space which basically follows from \cite{fasel:ij}. The inductive step compares the cohomology of the Grassmannians $\op{Gr}(k-1,n)$ and $\op{Gr}(k,n)$ via a space which appears as sphere bundle of the complementary resp. tautological bundle over $\op{Gr}(k-1,n)$ and $\op{Gr}(k,n)$, respectively. The image of $\beta$ is detected on the mod 2 Chow ring, which determines the multiplication with torsion classes.

\subsection{Localization sequence for inductive proof}

As a preparation for the inductive computation of $\mathbf{W}$-cohomology, we set up the relevant localization sequences which allow to compare cohomology of different Grassmannians.

\begin{proposition}
\label{prop:geomgrass}
\begin{enumerate}
\item There are isomorphisms
\[
\op{Gr}(k,n)\cong\op{Gr}(n-k,n).
\]
\item
Denote by $q\colon \mathscr{S}_k\to\op{Gr}(k,n)$ and $p\colon \mathscr{Q}_{n-k+1}\to\op{Gr}(k-1,n)$ the respective tautological bundles. Then there is an $\mathbb{A}^1$-weak equivalence of associated sphere bundles
\[
\mathscr{S}_k\setminus \op{Gr}(k,n)\simeq \mathscr{Q}_{n-k+1}\setminus\op{Gr}(k-1,n). 
\]
Moreover, under this weak equivalence, we have a correspondence of pullbacks of tautological vector bundles $q^\ast \mathscr{S}_k\cong p^\ast \mathscr{S}_{k-1}\oplus\mathscr{O}$. 
\end{enumerate}
\end{proposition}

\begin{proof}
(1) For a vector space $V$ of dimension $n$, the natural bijection between $k$-dimensional subspaces of $V$ and $(n-k)$-dimensional subspaces of its dual $V^\vee$ induces a natural isomorphism $\op{Gr}(k,V)\cong\op{Gr}(n-k,V^\vee)$. Choosing an isomorphism $V\cong V^\vee$ induces an isomorphism $\op{Gr}(n-k,V^\vee)\cong\op{Gr}(n-k,V)$. This provides the claimed isomorphisms. Note that these are not natural.

(2) Since we only want to establish an $\mathbb{A}^1$-weak equivalence, we can replace the Grassmannians $\op{Gr}(k,n)$ and $\op{Gr}(k-1,n)$ by the quotients $\op{GL}_n/\op{GL}_k\times\op{GL}_{n-k}$ and $\op{GL}_n/\op{GL}_{k-1}\times\op{GL}_{n-k+1}$, respectively. The pullback of the vector bundle $\mathscr{S}_k$ over $\op{GL}_n/\op{GL}_k\times\op{GL}_{n-k}$ is the associated bundle for the Stiefel variety $\op{GL}_n/\op{GL}_{n-k}$ viewed as $\op{GL}_k$-torsor and the natural $\op{GL}_k$-representation on $\mathbb{A}^k$. As in the setup of the localization sequence before Proposition~\ref{prop:locgln}, the complement of the zero section is then, up to a torsor under a unipotent group, $\op{GL}_n/\op{GL}_{k-1}\times\op{GL}_{n-k}$ because $\op{GL}_{k-1}$ is the stabilizer of a line in $\mathbb{A}^k$. A similar argument works for $\op{Gr}(k-1,n)$. The vector bundle $\mathscr{Q}_{n-k+1}$ over $\op{GL}_n/\op{GL}_{k-1}\times\op{GL}_{n-k+1}$ is the associated bundle for the Stiefel variety $\op{GL}_n/\op{GL}_{k-1}$ and the natural representation of $\op{GL}_{n-k+1}$ on $\mathbb{A}^{n-k+1}$. The complement of the zero section can then, up to a torsor under a unipotent group, be identified with $\op{GL}_n/\op{GL}_{k-1}\times\op{GL}_{n-k+1}$. This yields the required $\mathbb{A}^1$-weak equivalence. By an argument similar to the one in the setup for the localization sequence before Proposition~\ref{prop:locgln}, the pullback of the universal bundle $\mathscr{S}_k$ to $\mathscr{S}_k\setminus\op{Gr}(k,n)$ will split off a direct summand, and the remainder is the tautological rank $(k-1)$-bundle on $\op{GL}_n/\op{GL}_{k-1}\times\op{GL}_{n-k}$. On the other hand, the pullback of the tautological rank $(k-1)$-bundle on $\op{GL}_n/\op{GL}_{k-1}\times\op{GL}_{n-k+1}$ to $\op{GL}_n/\op{GL}_{k-1}\times\op{GL}_{n-k}$ will still be the tautological rank $(k-1)$-bundle. This proves the claims. 
\end{proof}

\begin{remark}
In abuse of notation, we will denote the total spaces of both sphere bundles $\mathscr{Q}_{n-k+1}\setminus\op{Gr}(k-1,n)$ and $\mathscr{S}_k\setminus\op{Gr}(k,n)$ by $\mathscr{S}(k,n)$. This is justified by Proposition~\ref{prop:geomgrass} because they are $\mathbb{A}^1$-equivalent and hence they will have isomorphic $\mathbf{W}$-cohomology. 
\end{remark}

We obtain two localization sequences, relating the Grassmannians $\op{Gr}(k,n)$ and $\op{Gr}(k-1,n)$ to their associated sphere bundles (which are weakly equivalent). This is the relevant input for the induction step for the computation of the $\mathbf{W}$-cohomology of the Grassmannians. 

\begin{proposition}
\label{prop:locgrass}
\begin{enumerate}
\item For any line bundle $\mathscr{L}$ on $\op{Gr}(k-1,n)$, there is a long exact localization sequence 
\begin{eqnarray*}
\cdots &\xrightarrow{\op{e}^\perp_{n-k+1}}& \op{H}^\bullet(\op{Gr}(k-1,n),\mathbf{W}(\mathscr{L})) \xrightarrow{p^\ast}  \op{H}^\bullet(\mathscr{S}(k,n),\mathbf{W}(\mathscr{L})) \\ &\xrightarrow{\partial}& \op{H}^{\bullet-n+k}(\op{Gr}(k-1,n),\mathbf{W}(\mathscr{L}\otimes\det\mathscr{Q}_{n-k+1})) \\&\xrightarrow{\op{e}_{n-k+1}^\perp}& \op{H}^{\bullet+1}(\op{Gr}(k-1,n),\mathbf{W}(\mathscr{L})) \to \cdots 
\end{eqnarray*}
\item For any line bundle $\mathscr{L}$ on $\op{Gr}(k,n)$, there is a long exact localization sequence 
\begin{eqnarray*}
\cdots \xrightarrow{\op{e}_{k}} \op{H}^\bullet(\op{Gr}(k,n),\mathbf{W}(\mathscr{L}))&\xrightarrow{q^\ast}& \op{H}^\bullet(\mathscr{S}(k,n),\mathbf{W}(\mathscr{L})) \to \\ \xrightarrow{\partial} \op{H}^{\bullet-k+1}(\op{Gr}(k,n),\mathbf{W}(\mathscr{L}\otimes\det\mathscr{S}_k))&\xrightarrow{\op{e}_k}& \op{H}^{\bullet+1}(\op{Gr}(k,n),\mathbf{W}(\mathscr{L})) \to \cdots 
\end{eqnarray*}
\end{enumerate}
\end{proposition}

Similar localization sequences are true for the other cohomology theories considered in this paper, but we will not need those.

\subsection{Inductive computation of $\mathbf{W}$-cohomology}

We now determine the structure of  the total $\mathbf{W}$-cohomology ring of $\op{Gr}(k,n)$. The argument completely follows the computation of rational cohomology of $\op{Gr}_k(\mathbb{R}^n)$ in \cite{sadykov}. Some formulas for oriented Grassmannians related to the ones below can already be found in Ananyevskiy's computation for $\eta$-inverted theories, cf.  \cite{ananyevskiy}.

\begin{theorem}
  \label{thm:invertedeta}
  Let $F$ be a  perfect field of characteristic $\neq 2$ and let $1\leq k<n$. The total $\mathbf{W}$-cohomology ring $\bigoplus_{i,\mathscr{L}}\op{H}^i(\op{Gr}(k,n),\mathbf{W}(\mathscr{L}))$ has the following presentation, as commutative $\mathbb{Z}\oplus\op{Pic}(\op{Gr}(k,n))/2$-graded $\op{W}(F)$-algebra:
\begin{enumerate}
\item For $k,n$ even, 
\[
\bigoplus_{j,\mathscr{L}}\op{H}^j(\op{Gr}(k,n),\mathbf{W}(\mathscr{L}))\cong \frac{\op{W}(F)[\op{p}_2,\dots,\op{p}_k,\op{e}_k,\op{p}_2^\perp,\dots,\op{p}_{n-k}^\perp,\op{e}_{n-k}^\perp]}{(\op{p}\cdot\op{p}^\perp=1,\op{e}_k\cdot\op{e}_{n-k}^\perp=0)}
\]
\item If $n$ is odd, 
\[
\bigoplus_{j,\mathscr{L}}\op{H}^j(\op{Gr}(k,n),\mathbf{W}(\mathscr{L}))\cong \left\{\begin{array}{ll} \frac{\op{W}(F)[\op{p}_2,\dots,\op{p}_k,\op{e}_k,\op{p}_2^\perp,\dots,\op{p}_{n-k-1}^\perp]}{(\op{p}\cdot\op{p}^\perp=1)} & k\textrm{ even}\\
\frac{\op{W}(F)[\op{p}_2,\dots,\op{p}_{k-1},\op{p}_2^\perp,\dots,\op{p}_{n-k}^\perp,\op{e}_{n-k}^\perp]}{(\op{p}\cdot\op{p}^\perp=1)} & k\textrm{ odd}
\end{array}\right.
\]
\item  For $k,n-k$ odd, 
\[
\bigoplus_{j,\mathscr{L}}\op{H}^j(\op{Gr}(k,n),\mathbf{W}(\mathscr{L}))\cong \frac{\op{W}(F)[\op{p}_2,\dots,\op{p}_{k-1},\op{p}_2^\perp,\dots,\op{p}_{n-k-1}^\perp]}{(\op{p}\cdot\op{p}^\perp=1)}\otimes \bigwedge[R]
\]
\end{enumerate}
Here the notation is the one of Theorem~\ref{thm:ingrass}, i.e., the bidegrees of the even Pontryagin classes $\op{p}_{2i}$ are $(4i,0)$, the bidegrees of the Euler classes $\op{e}_k$ and $\op{e}_{n-k}^\perp$ are $(k,1)$ and $(n-k,1)$, respectively, and the class $\op{R}$ in the last case has bidegree $(n-1,0)$. 
\end{theorem}

\begin{remark}
It should be pointed out that the description of the $\mathbf{I}^\bullet$-cohomology ring in Theorem~\ref{thm:ingrass} is compatible with the above claims via the natural projection $\op{H}^j(\op{Gr}(k,n),\mathbf{I}^j(\mathscr{L}))\to\op{H}^j(\op{Gr}(k,n),\mathbf{W}(\mathscr{L}))$. Moreover, Theorem~\ref{thm:ingrass} implies Theorem~\ref{thm:invertedeta}. 
\end{remark}

The following is an analogue of  Proposition~\ref{prop:kercokerchowmod2} for the above cohomology; it will be used in the inductive proof of Theorem~\ref{thm:invertedeta}. 

\begin{proposition}
\label{prop:kercokereulereta} 
Let $1\leq k<n$. Consider the morphism 
\[
\op{e}^\perp_{n-k}\colon \op{H}^{\bullet-n+k}(\op{Gr}(k,n),\mathbf{W}(\mathscr{L}))\to \op{H}^\bullet(\op{Gr}(k,n),\mathbf{W}(\mathscr{L}\otimes\det\mathscr{S}_{n-k}^\vee))
\]
given by multiplication with the Euler class. 
\begin{enumerate}
\item The cokernel is the quotient of the cohomology algebra modulo the ideal $\langle\op{e}_{n-k}^\perp\rangle$. 
\item If $k\equiv n-k\equiv 0\bmod 2$, then the kernel of $\op{e}_{n-k}^\perp$ is the ideal $\langle \op{e}_k\rangle$. The cokernel is generated by the classes $\op{p}_2,\dots,\op{p}_k$, $\op{e}_k$, $\op{p}_2^\perp,\dots,\op{p}_{n-k-2}^\perp$ modulo the relations $\op{p}\cdot\op{p}^\perp=1$ and $\op{e}_k^2=\op{p}_k$. The classes in the kernel are products of $\op{e}_k$ with a class in the cokernel. 
\item If $k+1\equiv n-k\equiv 0\bmod 2$, then the cokernel is generated by the Pontryagin classes $\op{p}_2,\dots,\op{p}_{k-1},\op{p}_2^\perp,\dots,\op{p}_{n-k-2}^\perp$ modulo the relation $\op{p}\cdot\op{p}^\perp=1$. The kernel is the ideal $\langle\op{p}_{k-1}\op{e}_{n-k}^\perp\rangle$. 
\item If $n-k \equiv 1\bmod 2$, the multiplication map is $0$. Kernel and cokernel are the whole cohomology algebra.
\end{enumerate}
\end{proposition}

\begin{proof}
This follows directly from the explicit presentation of Theorem~\ref{thm:invertedeta}.
\end{proof}

\begin{proofof}{Theorem~\ref{thm:invertedeta}} 
Fix a natural number $n$. The claim for $\op{Gr}(k,n)$ is proved by induction on $k$. 

The base case is the case $\mathbb{P}^{n-1}=\op{Gr}(1,n)$, in which case the claim follows directly from the computations in \cite{fasel:ij} (realizing for instance $\op{H}^i(\mathbb{P}^n,\mathbf{W}(\mathscr{L}))\cong \op{H}^i(\mathbb{P}^n,\mathbf{I}^{i-1}(\mathscr{L}))$. In both cases there are only two non-trivial groups, one of them is $\op{H}^0(\mathbb{P}^{n-1},\mathbf{W})\cong\op{W}(F)$. If $n$ is even, then $\mathbb{P}^{n-1}$ is orientable, and the other non-trivial cohomology groups is $\op{H}^{n-1}(\mathbb{P}^{n-1},\mathbf{W})\cong\op{W}(F)$ (non-twisted coefficients), generated by the orientation class $\op{R}$. If $n$ is odd, the other non-trivial cohomology group is $\op{H}^{n-1}(\mathbb{P}^{n-1},\mathbf{W}(\det\mathscr{S}_1))\cong\op{W}(F)$, generated by $\op{e}_{n-1}^\perp$.\footnote{Note the similarity with the rational cohomology of the real projective space.}

For the inductive step, assume that Theorem~\ref{thm:invertedeta} holds for $\op{Gr}(k-1,n)$. We have to make a case distinction depending on parities of $k$ and $n$. 

\textbf{If $n-k+1$ and $k-1$ are even,} then the Euler classes $\op{e}_{k-1}$ and $\op{e}_{n-k+1}^\perp$ are non-zero. Cokernel and kernel of $\op{e}_{n-k+1}^\perp$ are described in Part (1) and (2) of Proposition~\ref{prop:kercokereulereta}. As an algebra over the image of the cokernel of $\op{e}_{n-k+1}^\perp$, the cohomology of $\mathscr{S}(k,n)=\mathscr{Q}_{n-k+1}\setminus\op{Gr}(k-1,n)$ is an exterior algebra, generated by 1 and the class $\op{R}$ in  degree $(n-1,0)$ which is a lift of $\op{e}_{k-1}$ along $\partial$. This follows from the localization sequence for the bundle $\mathscr{Q}_{n-k+1}$. 

For the second localization, for the bundle $\mathscr{S}_k$, we first note that the Euler classes $\op{e}_k$ and $\op{e}_{n-k}^\perp$ are zero. We check what we can say about the map $q^\ast$: we have the Pontryagin classes $\op{p}_2,\dots,\op{p}_{k-1}$, and these are mapped to their counterparts in the cohomology of $\mathscr{S}_k\setminus\op{Gr}(k,n)$, by Proposition~\ref{prop:geomgrass}. By exactness, all the classes in the image of the restriction morphism $q^\ast$ will have trivial image under the boundary map $\partial$. Also, the class $\op{R}$ from degree $(n-1,0)$ has image under $\partial$ in degree $(n-k,1)$; in the case at hand, $n-k$ is odd, so there are no nontrivial elements in this degree and therefore $\partial \op{R}=0$. 

Now we need to determine which classes have nontrivial image under $\partial$. The class $\op{e}_{k-1}$ from the cokernel of $\op{e}_{n-k+1}^\perp$ necessarily maps to $1$ under $\partial$. By the derivation property,  more generally a product $p\cdot\op{e}_{k-1}$ of the Euler class with a polynomial $p$ in the Pontryagin classes $\op{p}_2,\dots,\op{p}_{k-1}$ will map under $\partial$ to $p$, viewed as an element of the cohomology of $\op{Gr}(k,n)$. 

At this point, we see that the candidate presentation is surjective: the Pontryagin classes $\op{p}_i$ and $\op{p}_i^\perp$ are mapped to their counterparts in the cohomology of $\mathscr{S}(k,n)$; and the same is true for the class $\op{R}$. The only missing generator of the cohomology of $\mathscr{S}(k,n)$ is the Euler class $\op{e}_{k-1}$, but we saw above that this class has nontrivial boundary. Injectivity of the candidate presentation follows similarly. 

\textbf{If both $n-k+1$ and $k-1$ are odd,} then the Euler classes $\op{e}_{k-1}$ and $\op{e}_{n-k+1}^\perp$ are zero. In particular, via the first localization sequence for the bundle $\mathscr{Q}_{n-k+1}\to\op{Gr}(k-1,n)$, the cohomology of $\mathscr{S}(k,n)$ consists of two copies of the cohomology of $\op{Gr}(k-1,n)$; one of the copies obviously generated by $1$ in degree $(0,0)$, the other generated by a class in bidegree $(n-k,1)$ which is a lift of $1\in \op{H}^0$ along the boundary map. 

Now for the second bundle $\mathscr{S}_k\to\op{Gr}(k,n)$, both Euler classes $\op{e}_k$ and $\op{e}_{n-k+1}^\perp$ are nontrivial. We check what we can say about the restriction map $q^\ast$ in the corresponding localization sequence. In the cohomology of $\op{Gr}(k,n)$, we have the Pontryagin classes and these are mapped under $q^\ast$ to their counterparts in the cohomology of $\mathscr{S}(k,n)$. The class in bidegree $(n-k,1)$ (which arose as lift of 1 in the first localization sequence) lifts to the Euler class $\op{e}_{n-k}^\perp$. 

The class $\op{R}$ from the cohomology of $\mathscr{S}(k,n)$ has nontrivial boundary; its degree is $(n-1,0)$ and its image under the boundary map has degree $(n-k,1)$, so the class $\op{R}$ is mapped exactly to the Euler class $\op{e}_{n-k}^\perp$. 

Consequently, we see that the candidate presentation is bijective. 

\textbf{If $n-k+1$ is even and $k-1$ is odd,} then the Euler class $\op{e}_{n-k+1}^\perp$ is nontrivial. Cokernel and kernel of $\op{e}_{n-k+1}^\perp$ are described in Proposition~\ref{prop:kercokereulereta}. The cokernel is generated by the Pontryagin classes, and the kernel is the ideal $\langle \op{p}_{k-2}\op{e}_{n-k+1}^\perp\rangle$. The class $\op{p}_{k-2}\op{e}_{n-k+1}^\perp$ lives in degree $(n-k+3,1)$ and consequently lifts along the boundary map $\partial$ to a class in degree $(2n-3,0)$. 

Now for the second bundle $\mathscr{S}_k\to\op{Gr}(k,n)$, the Euler class $\op{e}_k$ is also nontrivial. We check what happens in the associated localization sequence. The cokernel of the multiplication by $\op{e}_k$ on the candidate presentation is generated by the Pontryagin classes which map to their counterparts in the cohomology of $\mathscr{S}(k,n)$. The kernel of the Euler class on the candidate presentation is the ideal generated by $\op{e}_k\op{p}_{n-k-1}^\perp$ in degree $(2n-k-2,1)$. 

The Pontryagin classes in the cokernel all map to their counterparts under the restriction map $q^\ast$. The class in degree $(2n-3,0)$ (which arose as a lift of $\op{p}_{k-2}\op{e}_{n-k+1}^\perp$ maps to $\op{e}_k\op{p}_{n-k-1}^\perp$ in degree $(2n-k-2,1)$. 

Consequently, we see that the description of the cohomology of $\op{Gr}(k,n)$, given in Theorem~\ref{thm:invertedeta}, is true if and only if it is true for $\op{Gr}(k-1,n)$. Therefore, this argument also settles the case where $n-k+1$ is odd and $k-1$ is even.
\end{proofof}

\subsection{Putting the pieces together}

We are now in the position to prove the theorems about the structure of $\mathbf{I}$-cohomology of the Grassmannians $\op{Gr}(k,n)$. 

\begin{proofof}{Proposition~\ref{prop:inredgrass}}
  For all characteristic classes except $\op{R}$ the claims on their reductions follow from Theorem~\ref{thm:glnin} (2). The injectivity of $\rho$ on the image of $\beta_{\mathscr{L}}$  follows from Theorem~\ref{thm:invertedeta}, in combination with Lemma~\ref{lem:wsplit}, via the splitting of $\mathbf{I}$-cohomology as direct sum of $\mathbf{W}$-cohomology and the image of $\beta_{\mathscr{L}}$.

  It remains to identify the reduction of $\op{R}$. This follows by tracing through the inductive proof of Theorem~\ref{thm:glnin}, noting that $\op{R}$ arises via boundary maps from Euler classes. 
\end{proofof}

\begin{proofof}{Theorem~\ref{thm:ingrass}}
Again, we use the splitting of $\mathbf{I}$-cohomology as direct sum of $\mathbf{W}$-cohomology and the image of $\beta_{\mathscr{L}}$ which follows from Theorem~\ref{thm:invertedeta}, in combination with Lemma~\ref{lem:wsplit}.

  The relations (1)--(4) claimed in the theorem are satisfied because they are already satisfied on the level of ${\op{B}}\op{GL}_n$, by Theorem~\ref{thm:glnin} and the Whitney sum formulas in Propositions~\ref{prop:chern} and \ref{prop:pwhitney}. The relation (5) involving $\op{R}$ has two components: the claim on multiplication with torsion classes follows the injectivity of $\rho$ on the image of $\beta$ given by Proposition~\ref{prop:inredgrass}. The claim $\op{R}^2=0$ in $\mathbf{W}$-cohomology follows from Theorem~\ref{thm:invertedeta}. The image of $\op{R}^2$ under the projection to $\op{Im}\beta$ can be computed in mod 2 Chow theory, where we have $\rho(\op{R}^2)=\overline{\op{c}}_{k-1}\overline{\op{c}}_{n-k}^\perp\overline{\op{c}}_k\overline{\op{c}}_{n-k-1}^\perp=0$. In particular, we get a well-defined map from the candidate presentation to the total $\mathbf{I}$-cohomology ring of $\op{Gr}(k,n)$. 
  
  To show that the generators listed in Theorem~\ref{thm:ingrass} generate the $\mathbf{I}$-cohomology ring we again first show that all the torsion classes in the image of $\beta$ are accounted for. Knowing the mod 2 Chow ring of the Grassmannians, cf.~Proposition~\ref{cor:chowmod2grass}, this follows as in the proof of Theorem~\ref{thm:glnin} by considering the image of $\op{Sq}^2_{\mathscr{L}}$. Then the surjectivity for $\mathbf{W}$-cohomology follows from Theorem~\ref{thm:invertedeta}. 

  To show injectivity, i.e., that all relations in the cohomology ring are accounted for, we note that the $\op{W}(F)$-torsion free part generated by the Pontryagin classes, as well as Euler classes or $\op{R}$ whenever applicable, has exactly the relations (2), (3) and (5), by Theorem~\ref{thm:invertedeta}. So it suffices to investigate relations among classes in the image of $\beta$. Since $\rho$ is injective on the image of $\beta_{\mathscr{L}}$, it suffices to show that all relations appearing in $\op{Ch}^\bullet(\op{Gr}(k,n))$ arise from those for ${\op{B}}\op{GL}_n$ and the Whitney sum formulas. This follows from the presentation of the mod 2 Chow rings in Proposition~\ref{cor:chowmod2grass} and Theorem~\ref{thm:glnin}. 
\end{proofof}

\subsection{An example}

We discuss the argument for non-orientable Grassmannians in the special case comparing $\mathbb{P}^4$ and $\op{Gr}(2,5)$. The following computation also indicates how one may go about to establish the formulas for $\mathbf{I}$-cohomology directly withouth the $\beta$-$\mathbf{W}$-decomposition. A complete version of this argument can be found in the first version of the present paper on the arXiv.

First, we consider the localization sequence associated to the tautological rank 4 bundle on $\mathbb{P}^4$ which has the form
\[
\cdots\to \op{H}^j(\mathbb{P}^4,\mathbf{I}^j(\mathscr{L}))\to \op{H}^j(T,\mathbf{I}^j(\mathscr{L}))\to \op{H}^{j-3}(\mathbb{P}^4,\mathbf{I}^{j-4}(\mathscr{L}(1)))\xrightarrow{\op{e}_4^\perp} \cdots
\]
where $T$ is the complement of the zero section of the rank $4$ bundle on $\mathbb{P}^4$. From the shape of the localization sequence, we see that there are isomorphisms 
\[
\begin{array}{ll}
\op{H}^j(\mathbb{P}^4,\mathbf{I}^j(\mathscr{L}))\cong \op{H}^j(T,\mathbf{I}^j(\mathscr{L})) & \textrm{ for } j\leq 2, \textrm{ and } \\ \op{H}^{j+3}(T,\mathbf{I}^{j+3}(\mathscr{L}))\cong \op{H}^{j}(\mathbb{P}^4,\mathbf{I}^{j-1}(\mathscr{L}(1))) & \textrm{ for } j\geq 2
\end{array}
\]

The complicated bit is given by two exact sequences. First, we have 
\[
0\to \op{H}^3(\mathbb{P}^4,\mathbf{I}^3(\mathscr{L}))\to \op{H}^3(T,\mathbf{I}^3(\mathscr{L})) \to \op{H}^0(\mathbb{P}^4,\mathbf{I}^{-1}(\mathscr{L}(1)))\to \op{H}^4(\mathbb{P}^4,\mathbf{I}^3(\mathscr{L}))
\]
In the case where $\mathscr{L}=\mathscr{O}$, then the first and third terms in the exact sequence are trivial and so is $\op{H}^3(T,\mathbf{I}^3)$. In the case where $\mathscr{L}=\mathscr{O}(1)$, the third term is $\op{W}(F)\cdot 1$ and the last term is $\op{W}(F)\cdot \op{e}_4^\perp$ so that multiplication with the Euler class $\op{e}_4^\perp$ is an isomorphism. Consequently, we have an isomorphism $\op{H}^3(T,\mathbf{I}^3(1))\cong \op{H}^3(\mathbb{P}^4,\mathbf{I}^3(1))\cong\mathbb{Z}/2\mathbb{Z}$, generated by $\op{e}_1^3$. 

The second exact sequence is
\[
\op{H}^0(\mathbb{P}^4,\mathbf{I}^0(\mathscr{L}(1)))\to \op{H}^4(\mathbb{P}^4,\mathbf{I}^4(\mathscr{L}))\to \op{H}^4(T,\mathbf{I}^4(\mathscr{L}))\to \op{H}^1(\mathbb{P}^4,\mathbf{I}^0(\mathscr{L}(1)))\to 0
\]
In the case where $\mathscr{L}=\mathscr{O}$, then the first and last terms in the exact sequence are trivial and we get an isomorphism $\op{H}^4(T,\mathbf{I}^4)\cong\op{H}^4(\mathbb{P}^4,\mathbf{I}^4)\cong \mathbb{Z}/2\mathbb{Z}$, generated by $\op{e}_1^4$. In the case where $\mathscr{L}=\mathscr{O}(1)$, the first morphism is multiplication by the Euler class which is an isomorphism. In particular, we get an isomorphism $\op{H}^4(T,\mathbf{I}^4(1))\to \op{H}^1(\mathbb{P}^4,\mathbf{I}^0(1))\cong 0$.  

Now we can consider the localization sequence for the tautological rank $2$ bundle on $\op{Gr}(2,5)$ which has the form
\[
\cdots\xrightarrow{\op{e}_2} \op{H}^j(\op{Gr}(2,5),\mathbf{I}^j(\mathscr{L}))\to \op{H}^j(T,\mathbf{I}^j(\mathscr{L}))\to \op{H}^{j-1}(\op{Gr}(2,5),\mathbf{I}^{j-2}(\mathscr{L}(1)))\xrightarrow{\op{e}_2} \cdots
\]
Because of cohomology vanishing in negative degrees, we have an isomorphism $\op{H}^0(\op{Gr}(2,5),\mathbf{I}^0(\mathscr{L}))\cong\op{H}^0(T,\mathbf{I}^0(\mathscr{L}))$, and we note that this is isomorphic to the respective cohomology of $\mathbb{P}^4$. Next, there is an exact sequence
\[
0\to \op{H}^1(\op{Gr}(2,5),\mathbf{I}^1(\mathscr{L}))\to \op{H}^1(T,\mathbf{I}^1(\mathscr{L})) 
\]
For $\mathscr{L}=\mathscr{O}$, the last group is trivial, implying triviality of $\op{H}^1(\op{Gr}(2,5),\mathbf{I}^1(\mathscr{L}))$. For $\mathscr{L}=\mathscr{O}(1)$, the last group is $\mathbb{Z}/2\mathbb{Z}$. The explicit generator $\beta_{\mathscr{O}(1)}(1)$ maps to a generator of the last group and this implies that $\op{H}^1(\op{Gr}(2,5),\mathbf{I}^1(\mathscr{L}))\cong\mathbb{Z}/2\mathbb{Z}$. 

For $\op{H}^2$, we have an exact sequence 
\begin{eqnarray*}
\op{H}^0(\op{Gr}(2,5),\mathbf{I}^0(\mathscr{L}(1)))&\to& \op{H}^2(\op{Gr}(2,5),\mathbf{I}^0(\mathscr{L}))\to\\
\to \op{H}^2(T,\mathbf{I}^2(\mathscr{L}))&\to& \op{H}^1(\op{Gr}(2,5),\mathbf{I}^0(\mathscr{L}(1)))
\end{eqnarray*}
For $\mathscr{L}=\mathscr{O}$, the outer groups are both zero and hence $\op{H}^2(\op{Gr}(2,5),\mathbf{I}^0(\mathscr{L}))\cong\mathbb{Z}/2\mathbb{Z}$. For $\mathscr{L}=\mathscr{O}(1)$, the first map is an isomorphism mapping $1$ to $\op{e}_2$. Note that only using the localization sequence at this point would require knowledge of the restriction morphism  $\op{H}^1(\op{Gr}(2,5),\mathbf{I}^2(\mathscr{L}))\to\op{H}^1(T,\mathbf{I}^2(\mathscr{L}))$ to show that $\op{H}^2(\op{Gr}(2,5),\mathbf{I}^2(\mathscr{O}(1)))$ is isomorphic to $\op{W}(F)$ and not a proper quotient. 

The remaining cohomology groups can be computed similarly, producing exactly the results from Example~\ref{ex:gr25}. 

\subsection{Remarks on oriented Grassmannians}

We shortly formulate the analogous results for the Chow--Witt rings of the oriented Grassmannians. Recall that the $\mathbb{A}^1$-fundamental group of the Grassmannians is $\bm{\pi}_1^{\mathbb{A}^1}(\op{Gr}(k,n))\cong \mathbb{G}_{\op{m}}$ since the Grassmannians are $\op{GL}_{n-k}$-quotients of the Stiefel varieties $\op{GL}_n/\op{GL}_k$ which are highly $\mathbb{A}^1$-connected. The oriented Grassmannians $\widetilde{\op{Gr}}(k,n)$ are the $\mathbb{A}^1$-universal covers of the Grassmannians $\op{Gr}(k,n)$. Explicitly, they are given as the complement of the zero section of the line bundle $\det\mathscr{S}_k$. For the Chow--Witt rings of the oriented Grassmannians $\widetilde{\op{Gr}}(k,n)$, we only have the trivial duality because they are $\mathbb{A}^1$-simply connected. 

We can formulate a result analogous to Theorem~\ref{thm:ingrass} for the oriented Grassmannians. The proof of the result proceeds exactly along the lines of the proofs for $\op{Gr}(k,n)$. Some results concerning the $\mathbf{W}$-cohomology of the oriented Grassmannians can be deduced from the work of Ananyevskiy in \cite{ananyevskiy}. 

\begin{theorem}
Let $F$ be a perfect field of characteristic $\neq 2$, and let $1\leq k<n$.
\begin{enumerate}
\item   There is a cartesian square of $\mathbb{Z}$-graded $\op{GW}(F)$-algebras: 
\[
\xymatrix{
\widetilde{\op{CH}}^\bullet(\widetilde{\op{Gr}}(k,n)) \ar[r] \ar[d] & \ker \partial \ar[d]^{\bmod 2}\\
\op{H}^\bullet(\widetilde{\op{Gr}}(k,n),\mathbf{I}^\bullet) \ar[r]_>>>>>>\rho & \op{Ch}^\bullet(\widetilde{\op{Gr}}(k,n))
}
\]
\item 
The cokernel of the Bockstein morphism 
\[
\beta\colon \op{CH}^j(\widetilde{\op{Gr}}(k,n))\to \op{H}^{j+1}(\widetilde{\op{Gr}}(k,n),\mathbf{I}^{j+1})
\]
is described exactly as in Theorem~\ref{thm:invertedeta}, except that there is no additional $\mathbb{Z}/2\mathbb{Z}$-grading and the Euler classes live in the cohomology with trivial duality. 
\item The reduction morphism 
\[
\rho\colon \op{H}^{j+1}(\widetilde{\op{Gr}}(k,n),\mathbf{I}^{j+1})\to \op{CH}^{j+1}(\widetilde{\op{Gr}}(k,n))
\]
is injective on the image of the Bockstein morphism $\beta$. In particular, the image of Bockstein can be determined from the Wu formula for the Steenrod squares on the mod 2 Chow ring of $\widetilde{\op{Gr}}(k,n)$. 
\end{enumerate}
\end{theorem}

\begin{remark}
A result like the above should be true for all flag varieties (at least in type A). The cokernel of the Bockstein should have the same presentation as the rational cohomology of the real realization (but of course as a $\op{W}(F)$-algebra). The Bockstein classes should all be detected on the mod 2 Chow ring so that the structure of the torsion can be determined just from the knowledge of the Steenrod squares. 
\end{remark}

\end{document}